\documentclass[12pt,a4paper]{article}%
\usepackage{amsmath}
\usepackage{graphicx}
\usepackage{epsfig}%
\usepackage{amsfonts}%
\usepackage{amssymb}
\usepackage{algorithm}
\usepackage{algorithmic}
\usepackage{color}

\newtheorem{theorem}{Theorem}[section]

\newtheorem{lemma}[theorem]{Lemma}

\newtheorem{remark}[theorem]{Remark}

\newenvironment{proof}[1][Proof]{\noindent \emph{#1.} }{\hfill \ 
\rule{0.5em}{0.5em}}

\makeatletter\@addtoreset{equation}{section}\makeatother
\makeatletter\@addtoreset{figure}{section}\makeatother
\makeatletter\@addtoreset{table}{section}\makeatother
\begin{document}

\title{A reduced basis approach for calculation of the
Bethe-Salpeter excitation energies using \\ low-rank tensor factorizations} 

\author{Peter Benner\thanks{Max Planck Institute for Dynamics of Complex Systems, Magdeburg
({\tt benner@mpi-magdeburg.mpg.de})} \and
        Venera Khoromskaia\thanks{Max Planck Institute for
        Mathematics in the Sciences, Leipzig;
        Max Planck Institute for Dynamics of Complex Systems, Magdeburg ({\tt vekh@mis.mpg.de}).}
        \and Boris N. Khoromskij\thanks{Max Planck Institute for
        Mathematics in the Sciences, Inselstr.~22-26, D-04103 Leipzig,
        Germany ({\tt bokh@mis.mpg.de}).}}

%

\date{}

\maketitle
\begin{abstract}
The Bethe-Salpeter equation (BSE) is a reliable model for estimating
the absorption spectra in molecules and solids on the basis of
accurate calculation of the excited states from first principles.
 This challenging task includes calculation of the BSE operator in terms of two-electron integrals
tensor represented in molecular orbital basis,
and introduces a complicated algebraic task of solving the arising large matrix eigenvalue problem.
The direct diagonalization of the BSE matrix is practically intractable due to $O(N^6)$
complexity scaling in the size of the atomic orbitals basis set, $N$.
In this paper, we present a new  approach to the computation of Bethe-Salpeter excitation
energies which can lead to relaxation of the numerical costs up to $O(N^3)$. 
The idea is twofold: first, the diagonal plus low-rank tensor approximations to the fully populated
blocks in the BSE matrix is constructed, enabling easier partial eigenvalue solver for a large 
auxiliary system relying only on matrix-vector multiplications with rank-structured matrices. 
And second, a small subset of eigenfunctions from the auxiliary eigenvalue problem 
is selected to build the Galerkin projection of the exact BSE system onto the reduced basis set.
We present numerical tests on BSE calculations for a number of molecules confirming
the $\varepsilon$-rank bounds for the blocks of BSE matrix.
The numerics indicates that the reduced BSE eigenvalue problem with  
small matrices enables calculation of the lowest part of the excitation spectrum
with sufficient accuracy.
\end{abstract}

\noindent\emph{Key words:}
 Bethe-Salpeter equation, Hartree-Fock equation, two-electron integrals, 
tensor decompositions, model reduction, reduced basis, truncated Cholesky factorization.  

\noindent\emph{AMS Subject Classification:}
 65F30, 65F50, 65N35, 65F10
 
\section{Introduction}\label{Introd:MP2}

In modern material science
there is a growing interest to  the {\it ab initio} computation of absorption spectra for molecules 
or surfaces of solids. Due to model limitations, the first principles DFT or the Hartree-Fock 
calculations do not  allow  reliable estimates for excitation energies of molecular structures.
One of the approaches providing means for calculation of the excited states in 
molecules and solids is based on the solution of the Bethe-Salpeter equation (BSE)
\cite{ReOlRuOni:02,OniReRu:02,PiRoGa:2013,RiTouSa1:13,RiTouSa:13}.
The alternative methods treat the problem by using the time-dependent DFT or Green's
function approach \cite{RunGross:84,GrossKohn_book:90,Cas_BSE:95,StScuFr:98,PiRoGa:2013,RiTouSa:13}. 
The approximate coupled cluster calculations 
of electronic excitation energies by using rank decompositions have been 
described in \cite{HoKoPaMa_BSE:13}.
 
The BSE model, originating from high energy physics and incorporating the many-body perturbation 
theory and the Green's function formalism, governs calculation of the excited  
states in a self-consistent way. 
The BSE approach leads to the challenging computational task on the solution of the eigenvalue 
problem for a large fully populated matrix, that is in general non-symmetric.
It is worth to note that the size of BSE matrix scales quadratically
in the size of the basis set, $O(N^2_b)$, used in \emph{ab initio} electronic structure calculations.
Hence the direct diagonalization is limited by $O(N_b^{6})$  complexity  
making the problem
computationally extensive already for moderate size molecules with the size of the atomic orbitals
basis set, $N_b \approx 100$.
Furthermore, the numerical calculation of the matrix elements, based on the 
precomputed two-electron integrals (TEI) in the Hartree-Fock molecular orbitals basis,
has the numerical cost that scales at least as $O(N_b^4)$.
In the case of non-periodic lattice structured compounds (e.g. nano-structures) 
the number of basis functions increases proportionally to the lattice size
that easily leads to intractable problems even for small lattices.
Hence, a procedure that relies entirely on multiplications of a governing BSE matrix 
with vectors is the only viable approach.

In this way, the numerical solution of the BSE eigenvalue problem introduces 
several computationally extensive sub-problems. The commonly used approaches
for fast matrix computations are based on the use of certain  specific data-sparse structures 
in the target matrix. Here, we do not consider the sparsity based on the truncation
of small elements under the given threshold, since, in the particular case of BSE problem,
it seems hardly implementable because of the highly non-regular sparsity pattern arising.

In this paper we study the new approach to the solution of the BSE spectral problem based
on the model reduction via the reduced basis which is determined by the eigenvectors
of a simplified  system matrix of low-rank plus diagonal structure.
We investigate the approximation error of the reduced basis set depending on the
rank truncation parameters. The theoretical and numerical analysis of the existence of the low-rank approximation  
and the respective rank bounds for different matrix blocks in  the BSE matrix is presented.

This approach includes the low-rank decomposition of the matrix blocks in 
the Bethe-Salpeter kernel using the Cholesky  factorization  
of the two-electron integrals (TEI) \cite{VeKhBoKhSchn:12,VeKhorMP2:13} represented in the 
Hartree-Fock molecular orbitals basis.
The simplified block decomposition in the BSE system matrix
is determined by the separation  rank of order $O(N_b)$, which enables compact storage
and fast matrix-vector multiplications in the framework of iterations on a subspace
for finding a small part of the spectrum. 
This opens the way to the rank-truncated arithmetics of the reduced complexity, $O(N_b^{3})$, 
that may gainfully complement the existing numerical approaches to the challenging BSE model.

Methods for solving partial eigenvalue problems for matrices with special sparse structure
have been intensively studied in the literature, see for example,
\cite{BunByeMehrm:92,BeMeXu:97,Kressner_Diss:04,FaKre:06,BeFa:08,NaPoSaad:13,LinLinetc_SelInv:11}
and references therein. 
Brief surveys of commonly used rank-structured tensor formats can be found in 
\cite{KoldaB:07,khor-survey-2011,dc-phd} and in references therein.

The main computational tasks in  the presented approach include the following steps:
\begin{itemize}
\item Precompute the TEI tensor in the Hartree-Fock molecular orbital basis in the form of low-rank factorization.

\item Setting up  matrix blocks in the BSE matrix that  includes solution 
 of the linear matrix equation with  the identity plus low-rank governing matrix and 
 low-rank right-hand side.
 
\item Compute the low-rank approximation to the selected sub-matrices in the BSE matrix subject to the chosen
threshold criteria. 
 
 \item Construct the reduced basis set composed from eigenvectors corresponding
 to several lowest eigenstates of the rank-structured approximation to BSE matrix from the previous step.
 
\item Project the initial BSE matrix onto the reduced basis and diagonalize the arising 
  moderate size Galerkin matrix.

 \item Select the essential part in the spectrum of the projected Galerkin matrix 
 and build the  predicted excitation energies and the respective eigenstates. 
\end{itemize}

The design of the efficient linear algebra
algorithms for fast solution of arising large eigenvalue problems with 
rank-structured matrix blocks will be the topic for future work.

The rest of the paper is organized as follows. Section \ref{RFFTEI:sMP2} outlines 
the truncated Cholesky decomposition
scheme for low-rank factorization of the two-electron integrals tensor in the 
Hartree-Fock molecular orbitals basis, 
that is the building block in the construction of the BSE matrix. 
Section \ref{sec:FI_PostHF} describes the algebraic computational scheme 
for evaluation of the entries in the BSE matrix, analyses low-rank structure 
in the different matrix blocks and describes the reduced basis approach. 
We also analyze numerically the error of the commonly used simplified BSE model, 
the so-called Tamm-Dancoff (TDA) equation.
The most numerically extensive 
part in computation of the BSE matrix blocks is reduced to finding the low-rank solution of the matrix
equation with the diagonal plus low-rank structure in the governing matrix.
Numerical tests indicate the convergence in the senior (lowest) excitation 
energies by increase of the separation ranks. Conclusion summarizes  the main
algorithmic and numerical features of the presented approach and outlines further prospects.

\section{Low-rank approximation of the two-electron integrals in Hartree-Fock calculus} 
\label{RFFTEI:sMP2}

\subsection{Cholesky decomposition of the TEI matrix} \label{ssec:TEI_MP2}

The numerical treatment of the two-electron integrals (TEI) is the main
bottleneck in the numerical solution of the Hartree-Fock equation and in DFT calculations 
for large molecules.

Given the atomic orbitals basis set $\{g_\mu \}_{1\leq \mu \leq N_b}$, $g_\mu\in H^1(\mathbb{R}^3) $,
and the associated two-electron integrals (TEI) tensor ${\textbf{B}}=[b_{\mu \nu \lambda \sigma}]$ 
(see (\ref{eqn:btensor}) in Appendix), the associated $N_b^2 \times N_b^2$  TEI matrix over the large index 
set ${\cal I}\times {\cal J}$, $I=J= {\cal I}_b \otimes {\cal I}_b$, with 
${\cal I}_b:=\{1,...,N_b\}$,
$$
B = mat({\textbf{B}}) =[b_{\mu \nu; \lambda \sigma}] \in \mathbb{R}^{N_b^2 \times N_b^2},
$$ 
is
obtained by matrix unfolding of a tensor $ {\textbf{B}} =[b_{\mu \nu \lambda \sigma}]$.
The TEI matrix $B$ is proven to be symmetric and positive definite.
The optimized Hartree-Fock calculations are based on the incomplete Cholesky decomposition 
\cite{BeLi:77:Chol_TEI,Wils:90-TEI,Higham_CHdec:90,BeKurs:13,ReHelLin:2012,PaHoSchSheMar:13},
of the symmetric and positive definite matrix $B$, 
\begin{equation} \label{eqn:BCholesky}
B\approx L L^T, \quad L\in  \mathbb{R}^{N_b^2\times R_B},  \quad R_B=O(N_b).
\end{equation}
For this computation we apply the new Cholesky decomposition scheme, 
see \cite{VeKhBoKhSchn:12,VeKhorMP2:13}, where the adaptively chosen 
column vectors are calculated in the efficient way by using the precomputed 
redundancy free factorization of the TEI matrix $B$ (counterpart of the density fitting scheme). 
This allows the partial decoupling of the index sets $\{\mu \nu\}$ and $\{\lambda \sigma\}$.

Notice that the Cholesky factorization (\ref{eqn:BCholesky}) can be written in the index form 
\begin{equation} \label{eqn:BCholind}
b_{\mu \nu ; \lambda \sigma}\approx \sum\limits_{k=1}^{R_B} L_k(\mu; \nu) L_k(\sigma;\lambda),
\end{equation}
where the second factor corresponds to the transposed matrix $L_k^T$.
Here  $L_k=L_k(\mu; \nu)$, $k=1,...,R_B$, denotes the $N_b\times N_b$ matrix unfolding of the  
vector $L(:,k)$ in the Cholesky factor $L\in \mathbb{R}^{N_b^2\times R_B}$. 

The results of various numerical experiments indicate  that the truncated Cholesky decomposition 
with the separation rank $O(N_b)$ ensures the satisfactory numerical 
precision $\varepsilon>0$ of order $10^{-5}$ -- $10^{-6}$. 
The refined rank estimate $O(N_b |\log \varepsilon |)$ was observed in numerical experiments for every molecular 
system we considered so far \cite{VeKhBoKhSchn:12,VeKhorMP2:13}.

In the standard quantum chemical implementations in the Gaussian-type atomic orbitals basis
 the numerically confirmed rank bound
$rank(B) \leq C_B N_b$ ($C_B$ is about several ones) allows to reduce the complexity of 
building up the Fock matrix $F$
to $O(N_b^3)$, which is by far dominated by computational cost for the exchange term $K(D)$,
see Appendix.

\subsection{Rank estimates for TEI matrix $V$ and numerical illustrations} 
\label{ssec:TEI_MP2_orb}

Given the complete set of Hartree-Fock molecular orbitals $\{C_p\in \mathbb{R}^{N_b}\}$,
i.e. the column vectors in the coefficients matrix $C\in \mathbb{R}^{N_b\times N_b}$, 
and the corresponding energies $\{\varepsilon_p\}$, $p=1,2,...,N_b$  
(see Appendix, where $\lambda_i$ correspond to $\varepsilon_p$). In commonly used notations,
$\{C_i\}$ and $\{C_a\}$ denote the occupied and virtual orbitals, respectively.

For BSE calculations, one has to transform the TEI tensor  ${\bf B}=[b_{\mu \nu \lambda \sigma}]$,
corresponding to the initial AO basis set,
to those represented in the molecular orbital (MO) basis, 
\begin{equation} \label{eqn:TEI_Mol}
{\bf V}=[v_{i a j b}]: \quad v_{i a j b}= \sum\limits_{\mu, \nu, \lambda, \sigma =1}^{N_b} 
C_{\mu i} C_{\nu a}  C_{\lambda j} C_{\sigma b} b_{\mu \nu,  \lambda \sigma}, \quad 
a,b,i,j \in \{1,...,N_b\}.
\end{equation}
The BSE calculations make use of the two subtensors of ${\bf V}$ specified by the index sets
${\cal I}_{o}:=\{1,...,N_{orb}\}$ and ${\cal I}_{v}:=\{N_{orb}+1,...,N_{b}\}$, with $N_{orb}$ 
denoting the number of occupied orbitals (see Appendix). 
The first subtensor is defined as in the case of MP2 calculations,
\begin{equation} \label{eqn:TEI_Mol_MP2}
{\bf V}=[v_{i a j b}]: \quad a, b \in {\cal I}_{v},\quad i,j\in {\cal I}_{o},
\end{equation}
while the second one lives on the extended index set
\begin{equation} \label{eqn:TEI_Mol_ext}
{\bf V}=[v_{t u r s}]: \quad r, s \in {\cal I}_{v},\quad t,u\in {\cal I}_{o}.
\end{equation}
In the following, we shall use the notation 
$
N_{v}= N_b - N_{orb}, \; N_{ov}=N_{orb} N_{v}.
$

Denote the associated matrix by ${V}=[v_{i a, j b}]\in \mathbb{R}^{N_{ov}\times N_{ov}}$ 
in case (\ref{eqn:TEI_Mol_MP2}), 
and similar by ${V}=[v_{tu, rs}]\in \mathbb{R}^{N_{o}^2\times N_{v}^2}$ in case (\ref{eqn:TEI_Mol_ext}).
The straightforward computation of the matrix $ V$ by above representations makes the 
dominating impact to the overall 
numerical cost of order $O(N_b^5)$ in the evaluation of the block entries in the BSE matrix. 
The method of complexity $O(N_b^4)$ based on the low-rank
tensor decomposition of the matrix $V$ was introduced in \cite{VeKhorMP2:13} (see \S 2.2).

It can be shown that the rank $R_B=O(N_b)$ approximation to the TEI matrix 
$
B\approx L L^T,
$ 
with the $N\times R_B$ Cholesky factor $L$, allows to introduce the low-rank representation 
of the tensor ${\bf V}$, and then reduce the asymptotic complexity of 
calculations to $O(N_b^4)$, see \cite{VeKhorMP2:13}.
Indeed, let $C_m$ be the m-th column of the coefficient matrix 
$C=\{C_{\mu i}\}\in \mathbb{R}^{N_b \times N_b}$. 
Then substitution of (\ref{eqn:BCholind}) to (\ref{eqn:TEI_Mol}) in case (\ref{eqn:TEI_Mol_MP2})  leads to
\begin{align}\label{eqn:VChol}
 v_{i a j b}
& =
\sum\limits_{k=1}^{R_B} \sum\limits_{\mu, \nu, \lambda, \sigma =1}^{N_b} 
C_{\mu i} C_{\nu a}  C_{\lambda j} C_{\sigma b} L_k(\mu; \nu) L_k(\sigma;\lambda) \\ \nonumber
& = \sum\limits_{k=1}^{R_B} \left(\sum\limits_{\mu, \nu =1}^{N_b} C_{\mu i} C_{\nu a}L_k(\mu; \nu)\right)
\left(\sum\limits_{\lambda, \sigma =1}^{N_b}C_{\lambda j} C_{\sigma b}  L_k(\sigma;\lambda) \right) \\ \nonumber
& = \sum\limits_{k=1}^{R_B}  (C_i^T L_k C_a)   (C_b^T L_k^T C_j)
 = \sum\limits_{k=1}^{R_B}  (C_i^T L_k C_a)   (C_j^T L_k C_b)^T. 
\end{align}
 Similar factorization can be derived in case (\ref{eqn:TEI_Mol_ext}).
The precise formulation is given by the following lemma \cite{VeKhorMP2:13}, 
which will be used in further considerations.
\begin{lemma}\label{lem:MP2.V}
Let the rank-$R_B$ Cholesky decomposition of the matrix $B$ be given by (\ref{eqn:BCholesky}), then
the matrix unfolding $V=[v_{i a ;  j b}]$  allows a rank decomposition with  $rank(V) \leq R_B$. 
The $R_B$-term representation of the matrix unfolding $V=[v_{i a ;  j b}]$ takes a form
\[
 V=L_V L_V^T,\quad L_V\in \mathbb{R}^{N_{ov} \times R_B},  
\]
where the column vectors are given by
$$
L_V((i-1)N_{vir} +a- N_{orb};k)= C_i^T L_k C_a,\quad k=1,...,R_B, 
\;\; a\in {\cal I}_v,\;\; i\in {\cal I}_o.
$$
 In case (\ref{eqn:TEI_Mol_ext}) we have $V=U_V W_V^T \in \mathbb{R}^{N_{o}^2\times N_{v}^2}$ with
$U_V \in \mathbb{R}^{N_{o}^2\times R_B}$ and $W_V \in \mathbb{R}^{N_{ov}^2\times R_B}$.
\end{lemma}
Representation (\ref{eqn:VChol}) indicates that it is necessary to compute and store 
the only $L_V$, $U_V$ and $W_V$ factors in the above rank-structured factorizations.

Lemma \ref{lem:MP2.V} provides the upper bounds on $rank(V)$ in the representation (\ref{eqn:VChol}) which
might be larger than that obtained by the $\varepsilon$-rank truncation. 
It can be shown that the $\varepsilon$-rank of the matrix $ V$ remains of the same magnitude as
those for the TEI matrix $B$ obtained by its $\varepsilon$-rank truncated Cholesky decomposition 
(see numerics in \S\ref{ssec:Factor_BSE}).

Numerical tests in \cite{VeKhorMP2:13}
indicate that the singular values of the TEI matrix $B$ decay exponentially as 
\begin{equation} \label{eqn:svTEI}
\sigma_k \leq C e^{- \frac{z}{N_b} k},
\end{equation}
where the constant $z>0$ depends weakly on the molecule configuration. 
If we define $R_B(\varepsilon)$ as the minimal number satisfying the condition 
\[
\sum\limits_{k=R_B(\varepsilon)+1}^{R_B} \sigma_k^2 \leq \varepsilon^2,
\]
then  estimate (\ref{eqn:svTEI}) leads to the $\varepsilon$-rank
bound $R_B(\varepsilon) \leq C N_b |\log \varepsilon|$, which will be postulated in the following discussion.

Our goal is to justify that $R_V(\varepsilon)$ increases only logarithmically  in $\varepsilon$, 
similar to the bound for $R_B(\varepsilon)$.
To that end we introduce the SVD decomposition of the matrix $B$,
\[
 B= U D_B U^T, \quad U\in \mathbb{R}^{N_b^2 \times R_B}, \; D_B\in \mathbb{R}^{R_B \times R_B},
\]
which can be written in the index form
\begin{equation} \label{eqn:svdTEI}
 b_{\mu \nu ; \lambda \sigma} = \sum\limits_{k=1}^{R_B} \sigma_k U_k(\mu; \nu) U_k(\sigma;\lambda),
\end{equation}
with $U_k=[U_k(\mu; \nu)]\in \mathbb{R}^{N_b \times N_b}$ and $\|U_k \|_F =1$, $k=1,...,R_B$.
\begin{lemma}\label{lem:SVofV}
 For given $\varepsilon>0$,  there exists a rank-$r$ approximation $V_r$
to the matrix $V$ such that $r\leq R_B(\varepsilon)$ and
\[
\|V_r -V \| \leq  C N_b \varepsilon |\log \varepsilon|,
\]
where the constant $C$ does not depend of $\varepsilon$.
\end{lemma}
\begin{proof}
We estimate the $R_B(\varepsilon)$-term truncation error by using representation (\ref{eqn:svdTEI}),
 \begin{align}\label{eqn:V_SVD}
 v_{i a j b}   
 & =
 \sum\limits_{k=1}^{R_B}\sigma_k  \sum\limits_{\mu, \nu, \lambda, \sigma =1}^{N_b} 
 C_{\mu i} C_{\nu a}  C_{\lambda j} C_{\sigma b} U_k(\mu; \nu) U_k(\sigma;\lambda) \\ \nonumber
& = \sum\limits_{k=1}^{R_B}  \sigma_k \left(\sum\limits_{\mu, \nu =1}^{N_b} C_{\mu i} C_{\nu a}U_k(\mu; \nu)\right)
\left(\sum\limits_{\lambda, \sigma =1}^{N_b}C_{\lambda j} C_{\sigma b}  U_k(\sigma;\lambda) \right) \\ \nonumber
& = \sum\limits_{k=1}^{R_B}  \sigma_k (C_i^T U_k C_a)   (C_b^T U_k^T C_j) 
 = \sum\limits_{k=1}^{R_B}\sigma_k  (C_i^T U_k C_a)   (C_j^T U_k C_b)^T,
\end{align}
which can be presented in the matrix form $V=\sum\limits_{k=1}^{R_B} \sigma_k V_k V_k^T $,
where $V_k(i;a)= C_i^T U_k C_a$.
By definition of $R_B(\varepsilon)$ we have 
$\sum\limits_{k=R_B(\varepsilon)+1}^{R_B} \sigma_k^2 \leq \varepsilon^2$. 
Hence, the error of rank-$R_B(\varepsilon)$ approximation defined by 
$V_r = \sum\limits_{k=1}^{R_B(\varepsilon)} \sigma_k V_k V_k^T$, can be bounded by
\begin{align}\label{eqn:V_SVD_trunc}
\| \sum\limits_{k=R_B(\varepsilon)+1}^{R_B} \sigma_k V_k V_k^T \|
& \leq (\sum\limits_{k=R_B(\varepsilon)+1}^{R_B} \sigma_k^2)^{1/2} 
(\sum\limits_{k=R_B(\varepsilon)+1}^{R_B} \|V_k \|^4)^{1/2} \\ \nonumber
& \leq \varepsilon (\sum\limits_{k=R_B(\varepsilon)+1}^{R_B} \|V_k \|^4 )^{1/2} \\ \nonumber
& \leq \varepsilon (R_B - R_B(\varepsilon))\|C\|_{{\cal I}_o}^2 \|C\|_{{\cal I}_v}^2,
\end{align}
taking into account that $\|U_k  \|=1$, $k=1,...,R_B$, and the Frobenius norm estimate 
\[
\|V_k \|^2 = \|V_k(i;a) \|_F^2 = \|C_i^T U_k C_a\|_F^2 \leq  \|U_k \|^2 \sum\limits_{i,a} \|C_i\|^2 \|C_a\|^2
\leq \sum\limits_{i} \|C_i\|^2 \sum\limits_{a} \|C_a\|^2.  
\]
We suppose that $R_B= O(N_b |\log \varepsilon|)$, then
the multiple of $\varepsilon |\log \varepsilon|$ in (\ref{eqn:V_SVD_trunc}) does not depend 
on $\varepsilon$, that proves our lemma.
\end{proof}

The storage cost of these decompositions restricted to the active index set $I_{v} \times I_{o}$
amounts to $R_V(\varepsilon) N_{v} N_{orb}$. The complexity of straightforward computation 
on the active index set can be estimated by $O(R_B  N_{b}^2 N_{ov})$. 
\begin{figure}[htbp]
\centering
\includegraphics[width=5.2cm]{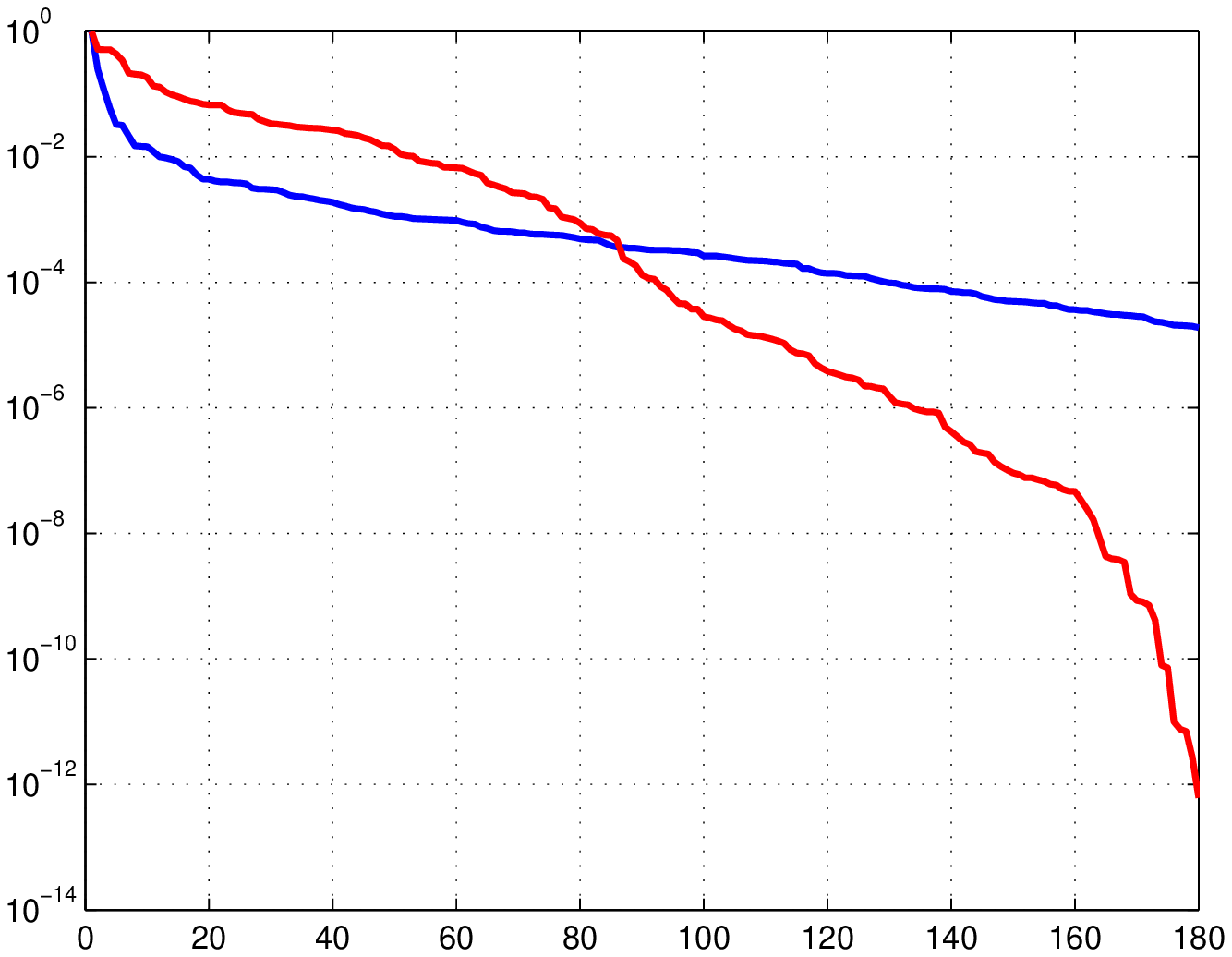}    
\includegraphics[width=5.2cm]{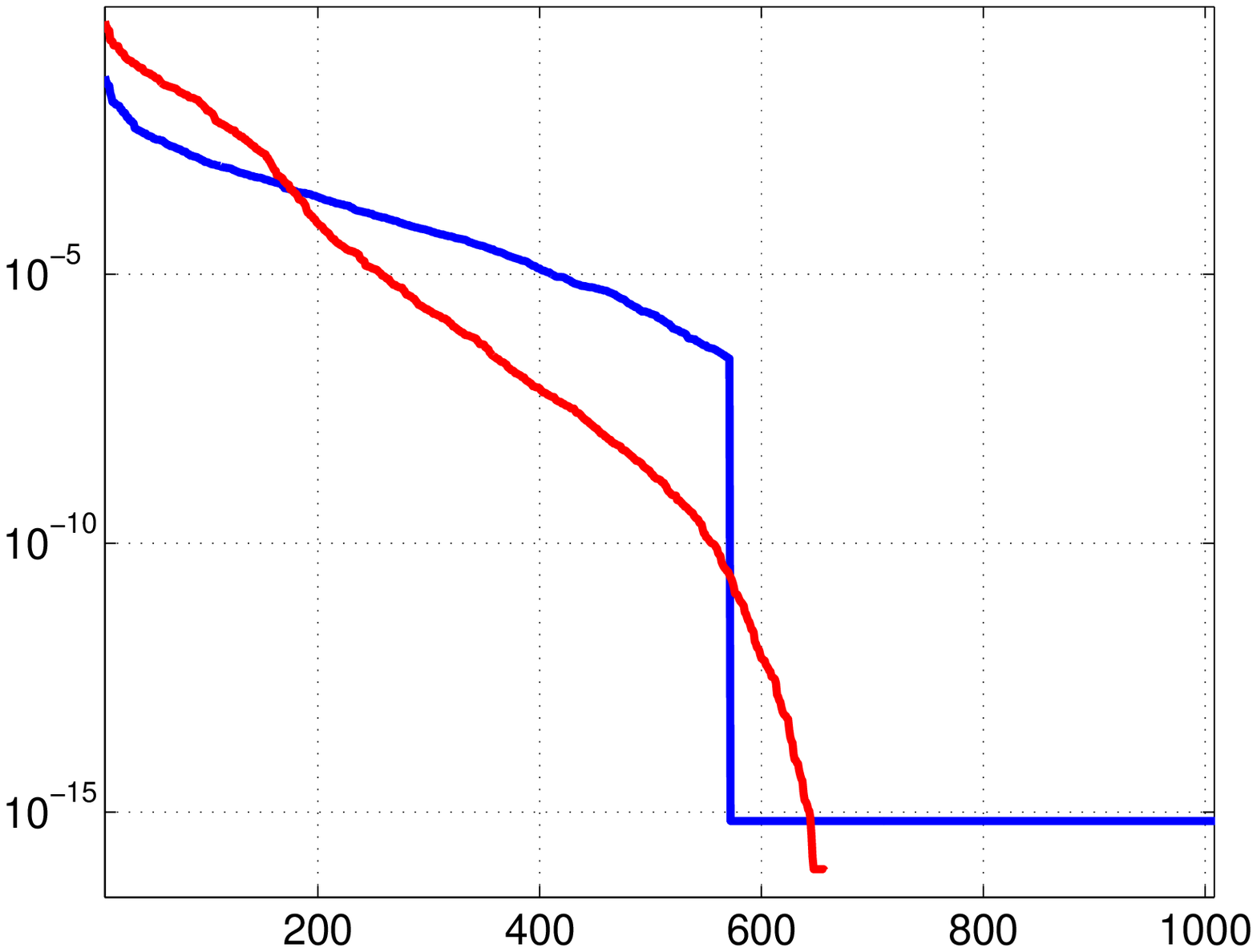}
\includegraphics[width=5.2cm]{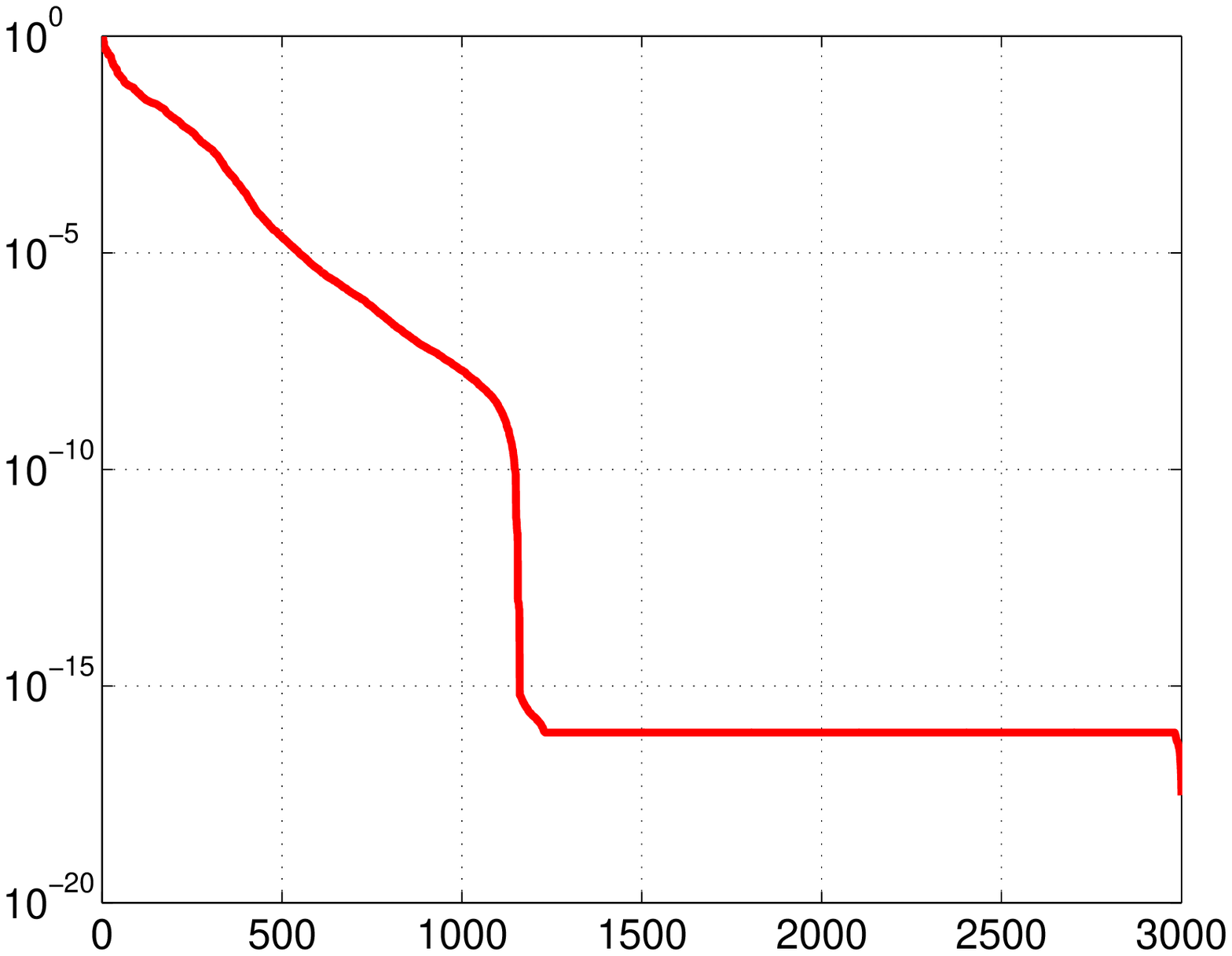}
\caption{\small Singular values of $V$ for several moderate size molecules.}
\label{fig:BSE_V_ranks}  
\end{figure}

Figure \ref{fig:BSE_V_ranks} represents the singular values of the matrices $V$ (red) and $B$ (blue)
for H$_2$O, N$_2$H$_4$, and C$_2$H$_5$NO$_2$ (Glycine amino acid)  molecules with the size
of the basis set $(N_b, N_{orb})$ equals to  
$(41, 9)$, $(82,9)$, and $(170,20)$, respectively. It can be seen that 
$R_V(\varepsilon)$ is linear proportional to $|\log \varepsilon|$ and it is of the same
order of magnitude as $R_B(\varepsilon)$.

These calculations are based on the reduced truncated SVD algorithm applied to 
the initial Cholesky decomposition of the matrix $V$ inherited from those for the TEI matrix $B$,
see Lemma \ref{lem:SVofV} and (\ref{eqn:V_SVD}).

\section{Block tensor factorization in BSE system matrix} \label{sec:FI_PostHF}

Here we discuss the main ingredients of the computational scheme for calculation of  blocks in
the BSE matrix and their reduced rank approximate representation.

\subsection{Tensor representations using TEI matrix in MTO basis} \label{ssec:DefMP2}

The construction of BSE matrix includes computation of several auxiliary quantities.
First, introduce a fourth order diagonal ''energy'' matrix by
\[
\boldsymbol{\Delta \varepsilon}=[\Delta \varepsilon_{ia,jb}]\in \mathbb{R}^{N_{ov}\times N_{ov}}: 
\quad \Delta \varepsilon_{ia,jb}= (\varepsilon_a - \varepsilon_i)\delta_{ij}\delta_{ab} ,
\]
that can be represented in the Kronecker product form
\[
\boldsymbol{\Delta \varepsilon} = 
I_{o}\otimes \mbox{diag}\{\varepsilon_a: a\in {\cal I}_v\} - 
\mbox{diag}\{\varepsilon_i: i\in {\cal I}_o\}\otimes I_{v},
\]
where $I_o$ and $I_v$ are the identity matrices on respective index sets.
It is worth to note that if the so called {\it homo lumo} gap of the system is positive,
i.e. 
$$
\varepsilon_a - \varepsilon_i > \delta >0,\quad a\in {\cal I}_v, i\in {\cal I}_o,
$$ 
then the matrix 
$\boldsymbol{\Delta \varepsilon}$ is invertible.

Using the  matrix $\boldsymbol{\Delta \varepsilon}$ 
and the $N_{ov}\times N_{ov}$ two-electron integrals matrix ${ V}=[v_{ia,jb}]$  
represented in the MO basis as in (\ref{eqn:TEI_Mol}),
the {\it dielectric function} ($N_{ov}\times N_{ov}$ matrix) $Z=[z_{pq,rs}]$ is defined by 
\[
 z_{pq,rs}:= \delta_{pr}\delta_{qs} - v_{pq,rs} [\boldsymbol{\chi}_0(\omega=0)]_{rs,rs},
\]
with $\boldsymbol{\chi}_0 (\omega)$ being the matrix form of the so-called Lehmann 
representation to the {\it response function}.
In turn, the matrix representation of $\boldsymbol{\chi}_0(\omega)$ inverse in known to
have a form
\[
\boldsymbol{\chi}_0^{-1}(\omega)= - 
\begin{pmatrix}
\boldsymbol{\Delta \varepsilon}   &  {\bf 0} \\ 
{\bf 0}  &  \boldsymbol{\Delta \varepsilon}  \\
\end{pmatrix} 
+ \omega 
\begin{pmatrix}
\boldsymbol{1}   &  {\bf 0} \\ 
{\bf 0}  &  -\boldsymbol{1}  \\
\end{pmatrix}, 
\]
implying
\[
\boldsymbol{\chi}_0 (0) = -
\begin{pmatrix}
\boldsymbol{\Delta \varepsilon}^{-1}   &  {\bf 0} \\ 
{\bf 0}  &  \boldsymbol{\Delta \varepsilon}^{-1}  \\
\end{pmatrix}. 
\]

Let ${\bf 1}\in \mathbb{R}^{N_{ov}}$ and 
${\bf d}_\varepsilon=\mbox{diag}\{\boldsymbol{\Delta \varepsilon}^{-1}\}\in \mathbb{R}^{N_{ov}}$
be the all-ones and diagonal vectors of $\boldsymbol{\Delta \varepsilon}^{-1}$, respectively,
specifying the rank-$1$ matrix ${\bf 1}\otimes {\bf d}_\varepsilon$.
In this notations the matrix $Z=[z_{pq,rs}]$ takes a compact form
\[
 Z = I_{o}\otimes I_{v} + V \odot \left( {\bf 1}\cdot {\bf d}_\varepsilon^T \right),
\]
where $\odot$ denotes the Hadamard product of matrices.
Introducing the inverse matrix $Z^{-1}$, we
finally, define the so-called {\it static screened interaction} matrix (tensor) by
\begin{equation} \label{eqn:BSE-Matr_W}
{ W}=[w_{pq,rs}]: \quad   
w_{pq,rs}:= \sum\limits_{t\in {\cal I}_{v},u\in {\cal I}_{o}} z^{-1}_{pq,tu} v_{tu,rs}.
\end{equation}
In the forthcoming calculations this equation should be considered on the conventional
and extended index sets
$\{p,s \in {\cal I}_{o}\} \cup \{ q,r \in {\cal I}_{v}\}$, and
$\{p,q \in {\cal I}_{o}\} \cup \{ r,s \in {\cal I}_{v}\}$, respectively,
corresponding to subtensors in (\ref{eqn:TEI_Mol_MP2}) and in (\ref{eqn:TEI_Mol_ext}).
%

Hence, on the conventional index set, we obtain the following matrix factorization 
of $W :=[w_{ia,jb}]$,
\[
 W = Z^{-1} V \quad \mbox{provided that} \quad a, b \in {\cal I}_{v},\quad i,j\in {\cal I}_{o}, 
\]
where $V$ is calculated by (\ref{eqn:TEI_Mol_MP2}),
while on the index set $\{p,q \in {\cal I}_{o}\} \cup \{ r,s \in {\cal I}_{v}\}$ the 
modified matrix ${W}=[{w}_{pq,rs}]$ is computed by (\ref{eqn:BSE-Matr_W})
and (\ref{eqn:TEI_Mol_ext}).

Now the matrix representation of the Bethe-Salpeter equation in the $(ov,vo)$ 
subspace reads as the following
eigenvalue problem determining  the excitation energies $\omega_n$:
\begin{equation} \label{eqn:BSE-GW1}
 F
 \begin{pmatrix}
 {\bf x}_n\\
{\bf y}_n\\
\end{pmatrix}
\equiv
\begin{pmatrix}
{ A}   &  { B} \\ 
{ B}^\ast  &  { A}^\ast  \\
\end{pmatrix} 
\begin{pmatrix}
 {\bf x}_n\\
{\bf y}_n\\
\end{pmatrix}
= \omega_n 
\begin{pmatrix}
{I}   &  { 0} \\ 
{ 0}  &  { -I}  \\
\end{pmatrix} 
\begin{pmatrix}
 {\bf x}_n\\ {\bf y}_n\\
\end{pmatrix},
\end{equation}
where the matrix blocks are defined in the index notation by
\[
 a_{ia,jb}:= \Delta \varepsilon_{ia,jb} + v_{ia,jb} - w_{ij,ab},
\]
\[
 b_{ia,jb}:=v_{ia,bj} - w_{ib,aj}, \quad a, b \in {\cal I}_{v},\quad i,j\in {\cal I}_{o}.
\]
In the matrix form we obtain
\[
 A=  \boldsymbol{\Delta \varepsilon} + V - \overline{W},
\]
where the matrix elements in $\overline{W}=[\overline{w}_{ia,jb}]$ are defined  by 
$\overline{w}_{ia,jb}=  {w}_{ij,ab} $, computed by  (\ref{eqn:BSE-Matr_W}) 
and (\ref{eqn:TEI_Mol_ext}).
Here the diagonal plus low-rank sparsity structure in $\boldsymbol{\Delta \varepsilon} + V$  
can be recognized  in view of Lemma \ref{lem:MP2.V}.  For the matrix block $B$ we have 
\[
B= \widetilde{V} - \widetilde{W} = {V} - \widetilde{W},
\]
where the matrix $\widetilde{V}$, corresponding to the partly transposed tensor, coincides with $V$, 
$$
\widetilde{\bf V}=[\widetilde{v}_{i a j b}]:=[v_{ia bj}] = [v_{ia jb}],
$$
and $\widetilde{W}$ is defined by permutation 
$\widetilde{W} = [ \widetilde{w}_{ia,jb} ]= [{w}_{ib,aj}] $.
In the following, we will investigate the reduced rank structure in the matrix blocks $A$ and $B$
resulted from the corresponding factorizations of $V$.

Solutions of equation (\ref{eqn:BSE-GW1}) come in pairs: excitation energies $\omega_n$ 
with eigenvectors $({\bf x}_n,{\bf y}_n)$, and de-excitation energies 
$-\omega_n$ with eigenvectors $({\bf x}_n^\ast,{\bf y}_n^\ast)$.

The block structure in matrices $A$ and $B$ is inherited from the symmetry of the 
TEI matrix $V$, $v_{ia,jb}= v^\ast_{ai,bj}$ and the matrix $W$, $w_{ia,jb}= w^\ast_{bj,ai}$.
In particular, it is well known from the literature that the matrix $A$ is Hermitian 
(since $v_{ia,jb}= v^\ast_{jb,ia}$ and $w_{ij,ab}= w^\ast_{ji,ba}$) and the matrix $B$ is symmetric 
(since $v_{ia,bj}= v_{jb,ai}$ and $w_{ib,aj}= w_{ja,bi}$).

In the following, we confine ourself by the case of the real spin orbitals, i.e. the 
matrices $A$ and $B$ remain real.
It is known that for the real spin orbitals and if 
${ A}   +  { B}$ and ${ A}   -  { B}$ are positive definite, the problem 
can be transformed into a half-size symmetric eigenvalue equation \cite{Cas_BSE:95}.
Indeed, in this case for every eigenpair we have,
\[
 A {\bf x} + B {\bf y} = \omega {\bf x}, \quad B {\bf x} + A {\bf y}= - \omega {\bf y},
\]
implying 
\[
(A+B)({\bf x}+{\bf y}) = \omega ({\bf x} - {\bf y}), \quad (A-B)({\bf x}-{\bf y}) = \omega ({\bf x} + {\bf y}).
\]
Now, if ${ A}   +  { B}$ and ${ A}   -  { B}$ are both positive definite, 
then the previous equations transform to
\begin{equation} \label{eqn:redBSE}
 M {\bf z} = \omega^2 {\bf z}\quad \mbox{with}  \quad M= (A - B)^{1/2}(A+B)(A - B)^{1/2},
\end{equation}
with respect to the normalized eigenvectors ${\bf z}= \sqrt{\omega}(A - B)^{1/2}({\bf x}+{\bf y}) $.
However, in this case the computation of the large fully populated matrix $(A - B)^{1/2}$ may become
the bottleneck.

The dimension of the matrix in (\ref{eqn:BSE-GW1}) is $2 N_o N_v \times 2 N_o N_v$, where $N_o$ and $N_v$ denote 
the number of occupied and virtual orbitals, respectively.  
In general, $N_o N_v$ is asymptotically of the size  $O(N_b^2)$, i.e. the spectral problem (\ref{eqn:BSE-GW1})
becomes computationally extensive already for moderate size molecules with $N_b \approx 100$.
Indeed, the direct eigenvalue solver for (\ref{eqn:BSE-GW1}) (diagonalization) becomes non-acceptable 
due to complexity scaling $O(N_b^{6})$.
Furthermore, the numerical calculation of the matrix elements, 
based on the precomputed TEI integrals from the Hartree-Fock equation, 
has the numerical cost that scales as $O(N_b^3) - O(N_b^5)$ depending on how to
compute the matrix $W$. Here again we propose to adapt the low-rank structure in the matrix $V$.

The most challenging computations arise in the case of lattice structured compounds, where the number of 
basis functions increases proportionally to the lattice size $L\times L \times L$, i.e. $N_b \sim n_0 L^3$,
that quickly leads to intractable problems even for small lattices.

\subsection{Eigenvalues in an interval by the low-rank approximation} \label{ssec:Factor_BSE}

The large matrix size in equation (\ref{eqn:BSE-GW1}) makes the solution of full eigenvalue problem
computationally intractable even for moderate size molecules, not saying for lattice structured compounds.
Hence, in realistic quantum chemical calculations of excitation energies the computation of several tens 
eigenpairs may be sufficient. 
Methods for solving partial eigenvalue problems and matrix inversion for large matrices 
with special sparsity pattern have been intensively studied in the literature, see for example, 
\cite{NaPoSaad:13}, \cite{LinLinetc_SelInv:11} and references therein. 

\subsubsection{The reduced basis approach by low-rank approximation} \label{sssec:Reduced_BSE}

In what following we show that the part $\boldsymbol{\Delta \varepsilon} + V$ in the matrix block $A$  
has the diagonal plus low-rank (DPLR) structure, while the sub-matrix $\widetilde{V}$ in 
the block $B$ exhibits the low-rank approximation. 
Taking into account these structures we propose the special  partial eigenvalue problems solver 
based on the use of reduced basis set obtained from as the eigenvectors of the reduced matrix that
picks up only the essential part of the initial BSE matrix with the DPLR structure. 
The iterative solver is based on fast matrix-vectors multiplication and efficient storage of 
all data involved into the computational scheme.
Using the reduced basis we than solve the initial problem by the Galerkin projection 
onto the reduced basis of moderate size.

Another direction that includes the QTT analysis
of the matrices involved in order to perform the fast matrix calculations in the QTT format will
be considered elsewhere.

We begin from the low-rank decomposition of the matrix $V$,
\[
 V \approx L_V L_V^T,\quad L_V\in \mathbb{R}^{N_{ov} \times R_V}, \quad R_V \leq R_B. 
\]
where the rank parameter $R_V =R_V(\varepsilon) =O(N_b |\log \varepsilon |)$ 
can be optimized depending on the truncation error $\varepsilon >0$ (see \cite{VeKhorMP2:13} 
and \S\ref{ssec:TEI_MP2_orb}).

First, we represent all matrix blocks and intermediate matrices included in the 
representation of the BSE matrix by using the above decomposition 
and diagonal matrices as follows. 
The properties of the Hadamard product imply that the matrix $Z$ exhibits the representation
\[
 Z = I_{o}\otimes I_{v} +  L_V L_V^T\odot \left( {\bf 1}\cdot {\bf d}_\varepsilon^T \right)
= I_{N_{ov}} + L_V (L_V \odot {\bf d}_\varepsilon)^T,
\]
where the rank of the second summand does not exceed $R_V$. Hence the matrix inversion $Z^{-1}$
within the calculation of $Z^{-1} V$
can be computed by special algorithm applied to the DPLR structure. 
The alternative way to compute the product $X= Z^{-1} V$ is the iterative solution of the matrix 
equation 
\begin{equation} \label{eqn:Comp_Zm1}
 Z X = V, \quad \mbox{with} \quad V=L_V {L}_V^T,
\end{equation}
and with the DPLR matrix $Z$. The above matrix equation  can be solved
in a low-rank format by using preconditioned iteration with rank truncation.

The computational cost for setting up the full BSE matrix $F$ in 
(\ref{eqn:BSE-GW1}) can be estimated by $O({N^3_{ov}})$, which includes the cost 
$O({N_{ov}}R_B)$ for 
generation of the matrix $V$ and the dominating cost $O({N^2_{ov}} N_{orb})$ for setting up of $W$.

In the following,  we rewrite spectral problem (\ref{eqn:BSE-GW1}) in the equivalent form
\begin{equation} \label{eqn:BSE-F1}
 F_1
 \begin{pmatrix}
 {\bf x}_n\\
{\bf y}_n\\
\end{pmatrix}
\equiv
\begin{pmatrix}
{ A}   &  { B} \\ 
-{ B}^\ast  &  -{ A}^\ast  \\
\end{pmatrix} 
\begin{pmatrix}
 {\bf x}_n\\
{\bf y}_n\\
\end{pmatrix}
= \omega_n 
\begin{pmatrix}
 {\bf x}_n\\ {\bf y}_n\\
\end{pmatrix}.
\end{equation}

The main idea of the {\it reduced basis approach} proposed in this paper is as follows.
Instead of solving the partial eigenvalue problem for finding of, say, 
$m_0$ eigenpairs in equation (\ref{eqn:BSE-F1}), we, first, solve the slightly 
simplified auxiliary spectral problem with a modified matrix $F_0$ 
obtained from $F_1$ by low-rank approximation of $\overline{W}$ and $\widetilde{W}$ from
the matrix blocks $A$ and $B$, respectively, i.e. by transforms
\begin{equation} \label{eqn:BSE-Reduce_Block}
A\mapsto A_0:=\boldsymbol{\Delta \varepsilon} + V -  \overline{W}_r \quad \mbox{and} 
\quad B \mapsto B_0: = {V} -  \widetilde{W}_r.
\end{equation}
Here we assume that the matrix $V$ is already presented in low-rank format, inherited 
from the Cholesky decomposition of TEI matrix.

%
%

The modified auxiliary problem reads
\begin{equation} \label{eqn:BSE-Reduced}
 F_0
 \begin{pmatrix}
 {\bf u}_n\\
{\bf v}_n\\
\end{pmatrix}\equiv
\begin{pmatrix}
{ A}_0   &  { B}_0 \\ 
-{B}_0^\ast  &  -{A}_0^\ast  \\
\end{pmatrix} 
\begin{pmatrix}
 {\bf u}_n\\
{\bf v}_n\\
\end{pmatrix}
= \lambda_n 
\begin{pmatrix}
 {\bf u}_n\\ {\bf v}_n\\
\end{pmatrix}.
\end{equation} 
This eigenvalue problem is much simpler than those in (\ref{eqn:BSE-GW1}) since
now the matrix blocks $A_0$ and $B_0$ are composed by diagonal and low-rank matrices.

Having at hand the set of $m_0 $  eigenpairs computed for the modified (reduced model)
problem (\ref{eqn:BSE-Reduced}), 
$\{(\lambda_n, \psi_n)=(\lambda_n,({\bf u}_n,{\bf v}_n)^T)\}$, we solve the full eigenvalue
problem for the reduced matrix obtained by projection of the initial equation onto the 
problem adapted small basis set $\{\psi_n\}$ of size $m_0$.

Define a matrix $G_1 = \psi_n(:,1:m_0)\in\mathbb{R}^{2N_{ov}\times m_0}$ whose columns present
the vectors of reduced basis, 
compute the Galerkin  and mass matrices by projection onto the reduced basis specified by columns in $G_1$, 
\[
M_1= G_1^T F_1 G_1, \quad S_1 = G_1^T G_1 \in \mathbb{R}^{m_0\times m_0},
\]
and then solve the projected generalized eigenvalue problem of small size $m_0\times m_0$,
\begin{equation} \label{eqn:BSE-Red-Galerk}
 M_1 Y = \gamma_n S_1 Y, \quad Y \in \mathbb{R}^{m_0 }.
\end{equation}
The portion of small eigenvalues $\gamma_n$, $n=1,...,m_0$, 
is thought to be very close to the corresponding excitation energies 
$\omega_n$, ($n=1,...,m_0$) in the initial spectral  problem (\ref{eqn:BSE-GW1}).
Table \ref{table_e_Nred} illustrates that the large the size of reduced basis  $m_0$,
the better the accuracy of the lowest excitation energy $\gamma_1$.
\begin{center}%
\begin{table}[htb]
\begin{center}
\begin{tabular}
[c]{|c|c|r|r|r|r|r|}%
\hline
 $m_0$     &   $5$    & $10$     & $20$     & $30$ & $40$    & $50$   \\
\hline\hline
H$_2$ O   &  $0.025$   & $0.025$  & $0.14$ & $0.01$   & $0.01$     & $0.005$  \\
\hline
N$_2$H$_4$ &  $0.02$   & $0.02$  & $0.015$ & $0.015$  & $0.015$     & $0.005$  \\
\hline\hline
 \end{tabular}
\end{center}
\caption{The error $|\gamma_1 - \omega_1|$ vs. the size of reduced basis, $m_0$.}
\label{table_e_Nred}
\end{table}
\end{center}

\begin{remark} \label{rem:Rank_W}
Notice that the matrix $\overline{W}$  might have rather large $\varepsilon$-rank for small values 
of $\varepsilon$ which increases the cost of high accuracy solutions.
The results of numerical tests that follow (see Table \ref{table_e_ranks}) indicate that the 
rank approximation to the matrix $\overline{W}$ with the moderate rank parameter allows for the 
numerical error in the excitation energies of the order of few percents. 
Further improvement of the accuracy requires noticeable increase in the computational costs.
To avoid this numerical payoff, we apply another approximation strategy in which 
the matrix $\overline{W}$ remains unchanged, while matrices $V$ and $\widetilde{W}$
are substituted by their low-rank approximation (see Figure \ref{fig:BSE_NH3_reduced}).
\end{remark}

Matrix blocks in the auxiliary equation (\ref{eqn:BSE-Reduced}) are obtained by rather rough $\varepsilon$-rank
approximation to the initial system matrix. However, we observe much better approximations $\gamma_n$
from (\ref{eqn:BSE-Red-Galerk}) to the exact excitation energies $\omega_n$ from 
the equation (\ref{eqn:BSE-GW1}). 
This can be explained 
by the well known effect of the quadratic error behavior of eigenvalues with respect to 
the perturbation error in the symmetric matrix.
In the situation with equation (\ref{eqn:redBSE}) the corresponding statement can 
be easily proved under mild assumptions.

\begin{lemma}\label{lem:ErrorLam}
Let matrices $A$ and $B$ be real and both $A-B $ and $A+B $ be symmetric, positive definite.
Suppose that the matrices in the  system (\ref{eqn:redBSE}) are perturbed by 
$A\mapsto \widetilde{A}:=A + \Delta A$, $B\mapsto \widetilde{B}:= B+ \Delta B$, such that 
$\|\Delta A\| \leq \varepsilon $ and $\|\Delta B\| \leq \varepsilon $.
Then the error in the excitation energies, 
 $\Delta \omega_n = \omega_n - \widetilde{\omega}_n$,   is estimated by
\[
 |\Delta \omega_n | \leq C \varepsilon^2,
\]
provided that $ |2 \omega_n + \Delta \omega_n|\geq  \delta >0 $ uniformly in $\varepsilon $.
\end{lemma}
\begin{proof}
 Denote by $\lambda_n = \omega_n^2$ and $\widetilde{\lambda}_n = \widetilde{\omega}_n^2$ 
 the exact and perturbed
eigenvalues in the  transformed  problem (\ref{eqn:redBSE}). This problem is symmetric, hence we have
\[
|\omega_n^2 - \widetilde{\omega}_n^2 | = | \lambda_n - \widetilde{\lambda}_n |  \leq C \varepsilon^2,
\]
implying
\[
 |\omega_n - \widetilde{\omega}_n | \leq \frac{C}{|\omega_n + \widetilde{\omega}_n |} \varepsilon^2,
\]
which proves the statement.
\end{proof}

In the particular BSE formulation based on the Hartree-Fock molecular orbitals basis, we have
the slightly perturbed symmetry in the matrix blocks, i.e. Lemma \ref{lem:ErrorLam} 
does not apply directly.
However, we observe the same quadratic error decay in all numerical experiments 
implemented so far.

\subsubsection{Numerics for the reduced basis methods} \label{sssec:Reduced_BSE_numer}

In this section we present numerical illustrations to the reduced basis
approach for the BSE problem,
which use the TEI tensor and molecular orbitals obtained from the solution 
of the Hartree-Fock equation
by the 3D grid-based tensor-structured method \cite{vekh:13,VeKhorMP2:13}.
All examples below utilize the grid representation of the Galerkin basis 
functions from Gaussian basis sets of type cc-pDVZ.
The two-electron integrals are computed in the form of low-rank Cholesky factorization
by tensor-structured algorithms
incorporating 1D density fitting \cite{VeKhBoKhSchn:12}.

In the following numerical tests we demonstrate on the examples of moderate size molecules that
a small reduced basis set, obtained by separable approximation
with the rank parameters of about several tens, allows to reveal several lowest
excitation energies and respective excited states with the accuracy about $0.1$eV - $0.02$eV
 depending on the rank-truncation strategy. Table \ref{tab:BasisParam} represents
the size of GTO basis set, $N_b$, and the number of molecular orbitals, $N_{orb}$, in numerical 
examples considered below.

\begin{center}%
\begin{table}[htb]
\begin{center}
\begin{tabular}
[c]{|c|c|r|r|r|r|}%
\hline
 &          H$_2$O  & H$_2$O$_2$ & N$_2$H$_4$ & C$_2$H$_5$OH & C$_2$H$_5$NO$_2$   \\
\hline\hline
$N_{orb}$, $N_b$ & $5$, $41$ & $9$, $68$ & $9$, $82$& $13$, $123$ & $20$, $170$ \\
\hline
 \end{tabular}
\end{center}
\caption{GTO basis set size $N_b$ and number of molecular orbitals, $N_{orb}$, in considered examples.}
\label{tab:BasisParam}
\end{table}
\end{center} 

\begin{center}%
\begin{table}[htb]
\begin{center}
\begin{tabular}
[c]{|c|c|r|r|r|r|}%
\hline
& $\varepsilon$ & $4\cdot 10^{-1}$  & $2\cdot 10^{-1}$ & $10^{-1}$   &  $10^{-2}$          \\
\hline\hline
H$_2$O   &  $|\gamma_1 - \omega_1|$ & $10^{-2}$  & $10^{-2}$ & $8\cdot 10^{-3}$ & $8\cdot 10^{-6}$    \\
\cline{2-6}
          &$|\lambda_1 - \omega_1|$ & $0.35$ & $0.35$ & $0.31$ & $4.4\cdot 10^{-4}$  \\
\cline{2-6}
          & $\|F_1 -F_0\|$ & $1.17$ &  $1.17$ & $8\cdot 10^{-1}$ & $2\cdot 10^{-2}$     \\
\cline{2-6}
   & \footnotesize{ranks $V$,$\overline{W}$,$\widetilde{W}$} & $6$, $9$, $6$ & $13$, $13$, $10$
   & $25$, $72$, $36$ & $60$, $180$, $92$   \\
\hline\hline
N$_2$H$_4$ & $|\gamma_1 - \omega_1|$ & $1.4\cdot 10^{-2}$ & $1.5\cdot 10^{-2}$ & $10^{-2}$
& $6\cdot 10^{-6}$    \\
\cline{2-6}
          &$|\lambda_1 - \omega_1|$ &$0.25$ & $0.25$ &$0.25$ &  $2.8\cdot 10^{-4}$     \\
 \cline{2-6}
          & $\|F_1 -F_0\|$ & $1.8$ &  $1.0$     & $6\cdot 10^{-1}$ & $1.4\cdot 10^{-2}$     \\
\cline{2-6}
& \footnotesize{ranks $V$,$\overline{W}$,$\widetilde{W}$} & $11$, $17$, $11$ & $26$, $25$, $15$ &
$49$, $144$, $54$ & $117$, $657$, $196$    \\
\hline \hline
\small{C$_2$H$_5$OH} & $|\gamma_1 - \omega_1|$ & $3\cdot 10^{-2}$ & $3\cdot 10^{-2}$
& $1.5\cdot 10^{-2}$ & $6\cdot 10^{-6}$    \\
\cline{2-6}
          &$|\lambda_1 - \omega_1|$ &$0.29$ & $0.29$ & $0.27$ & $3\cdot 10^{-4}$  \\
 \cline{2-6}
& $\|F_1 -F_0\|$ & $3.0$ &  $1.5$   & $7\cdot 10^{-1}$ & $1.4\cdot 10^{-2}$ \\
\cline{2-6}
& \footnotesize{ranks $V$,$\overline{W}$,$\widetilde{W}$} & $16$, $17$, $14$ & $39$, $29$, $20$ & $71$, $105$, $74$ & $171$, $1430$, $296$    \\
\hline \hline
\end{tabular}
\end{center}
\caption{Accuracy for 1-st eigenvalue, $|\gamma_1 - \omega_1|$, and norms of the difference between the
exact and reduced-rank
matrices, $\|F_1 -F_0\|$, vs. $\varepsilon$-rank for $V$, $\overline{W}$ and $\widetilde{W}$.
The BSE matrix size is given in brackets:
H$_2$O ($360\times 360$), N$_2$H$_4$ ($1430 \times 1430$), C$_2$H$_5$OH ($2860\times 2860$).
}
\label{table_e_ranks}
\end{table}
\end{center}

Table \ref{table_e_ranks} demonstrates the quadratic decay of the error $|\gamma_1 - \omega_1|$
in the lowest excitation energy with respect to the approximation error to the initial BSE matrix, 
which is controlled by a tolerance $\varepsilon$ in the rank  truncation procedure
applied to  the BSE submatrices $V$, $\overline{W}$ and $\widetilde{W}$. 
The resulting $\varepsilon$-ranks for the corresponding matrices 
are presented for H$_2$O, N$_2$H$_4$ and C$_2$H$_5$OH molecules.
 The error for $1$st eigenvalue, $|\gamma_1 - \omega_1|$, is given in Hartree (one Hartree corresponds
to $27$eV). This table demonstrates that the error in the reduced basis approximation, $|\gamma_1 - \omega_1|$,
is at least one order of magnitude smaller than those for simplified problem, $|\lambda_1 - \omega_1|$,
which motivates the use of the reduced basis equation (\ref{eqn:BSE-Red-Galerk}). 

This effect can be also seen in Figure \ref{fig:BSE_N2H4_reduced} 
demonstrating the convergence $\gamma_n \to \omega_n$ and $\lambda_n \to \omega_n$ 
with respect to the increasing rank parameter 
determining the auxiliary problem (the size of reduced basis is $m_0=30$).
\begin{figure}[htbp]
\centering
\includegraphics[width=5.4cm]{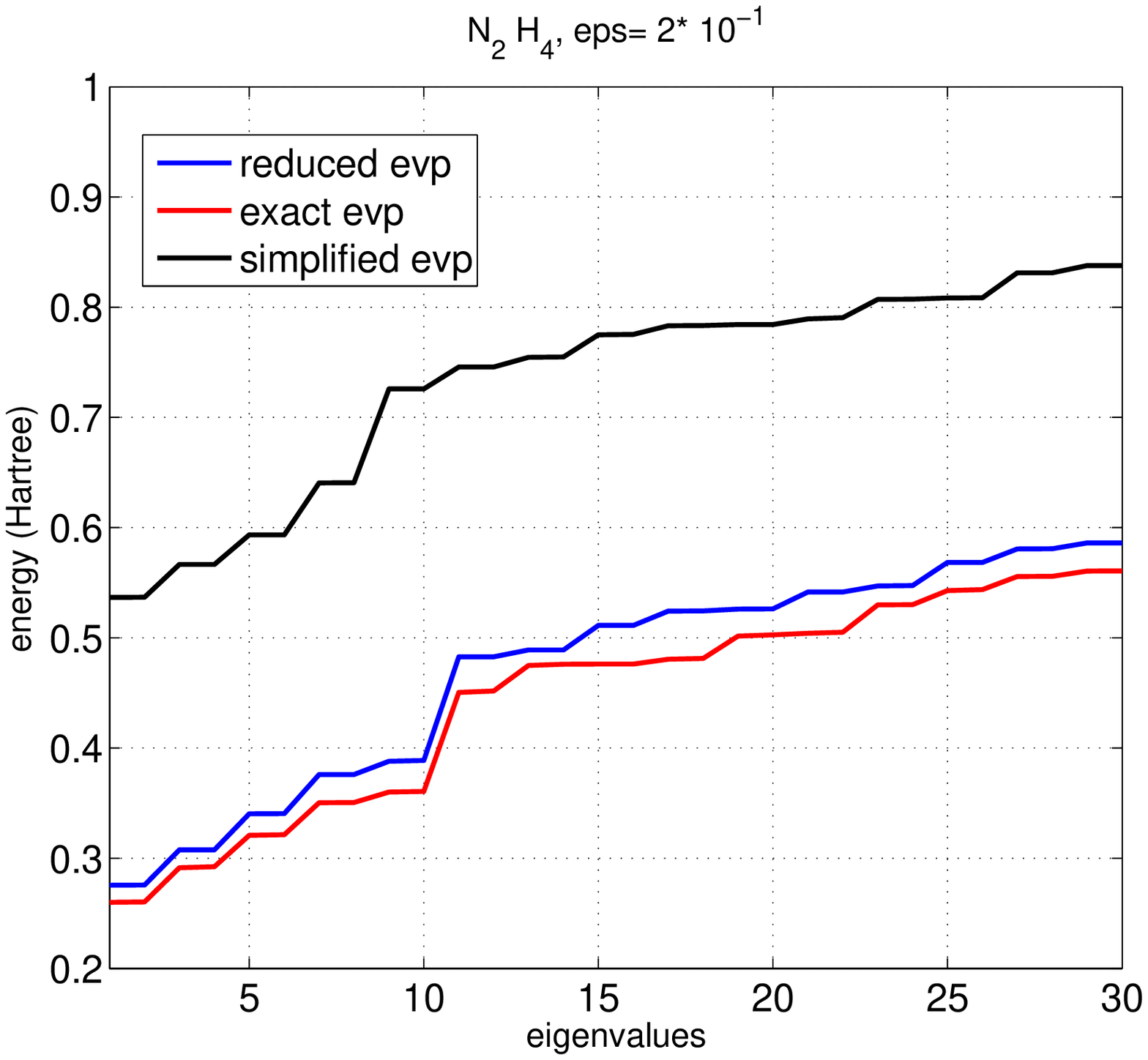}    
\includegraphics[width=5.4cm]{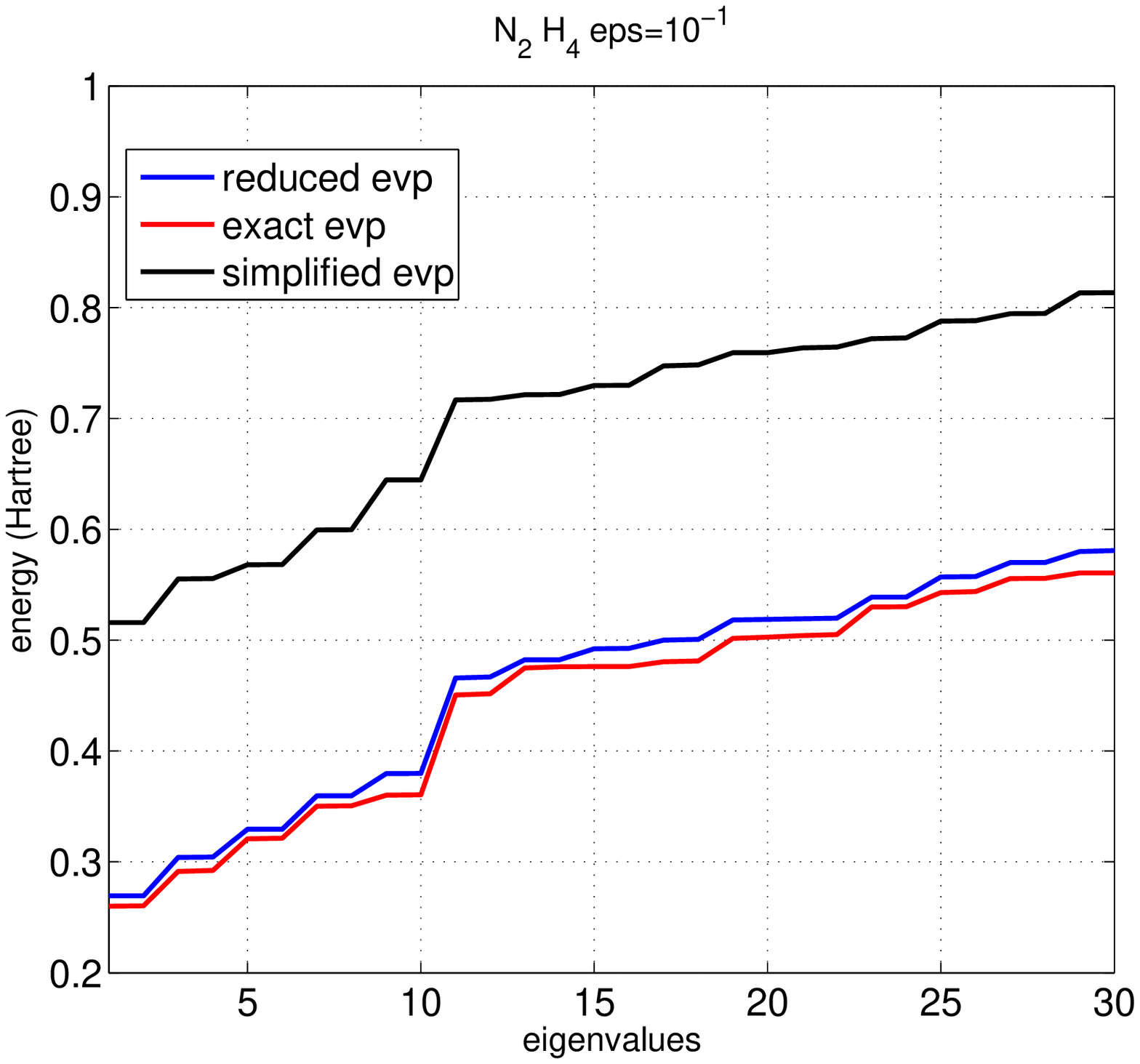}
\includegraphics[width=5.4cm]{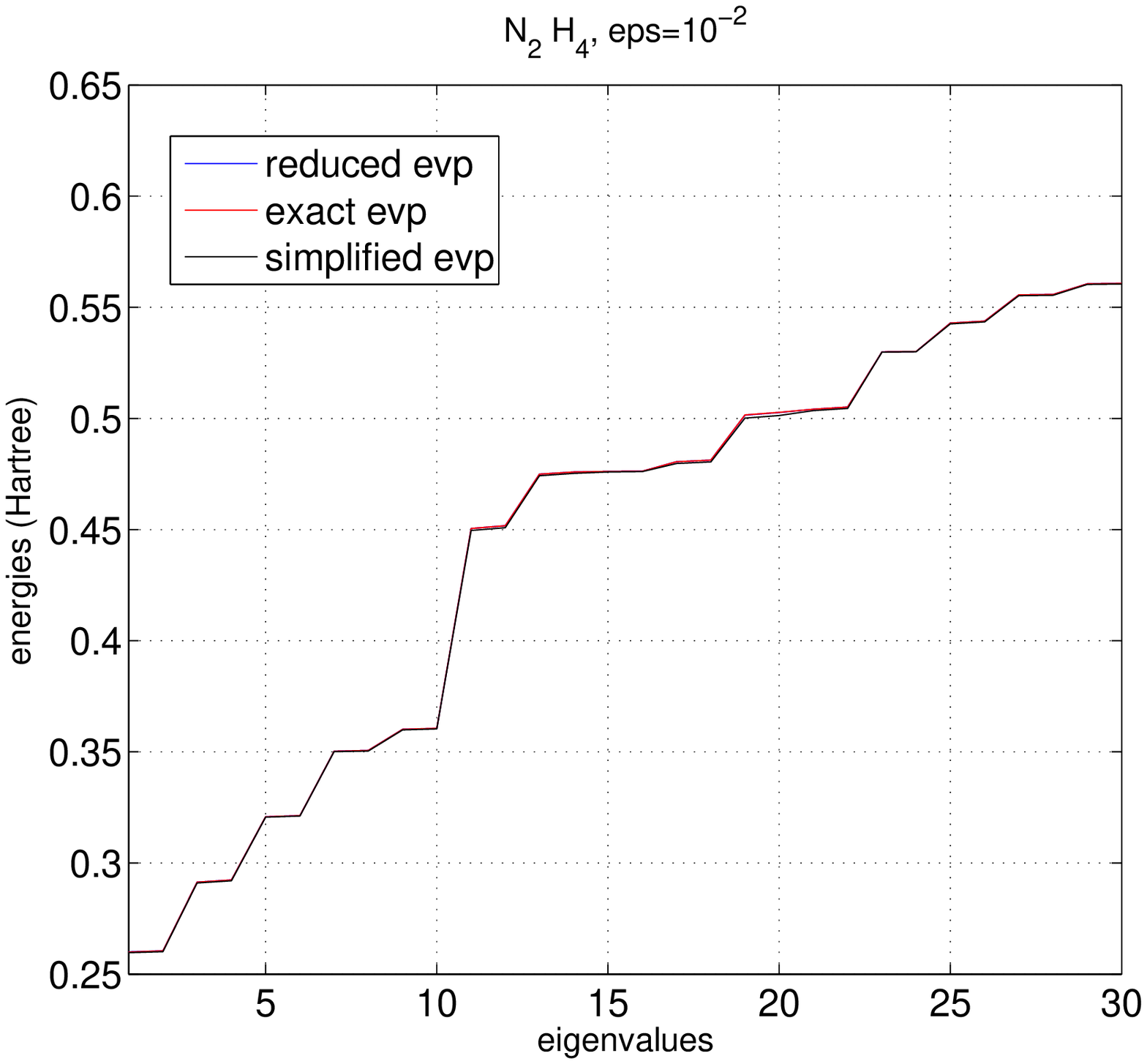}
\caption{\small Comparison of $m_0=30$ lower eigenvalues for the reduced and exact BSE 
systems vs. $\varepsilon$ in the case of N$_2$H$_4$ molecule.}
\label{fig:BSE_N2H4_reduced}  
\end{figure}
It confirms the numerical observation (see also Table \ref{table_e_ranks}) that 
\[
 |\gamma_n - \omega_n|\ll |\lambda_n - \omega_n|,
\]
that justifies the efficiency of the reduced basis approach.
Three figures from the left to the right correspond to the rank truncation threshold
$\varepsilon\in \{2\cdot 10^{-1}$, $10^{-1} $, $10^{-2}\}$. 
The quantities $\lambda_n$, $\gamma_n$ and $\omega_n$ are marked by black, blue and red lines,
respectively.
Notice that for $\varepsilon= 10^{-2}$ the energies (eigenvalues) for the initial and reduced systems  
are practically coinciding (error of the order of $10^{-6}$), 
at the expense of large separation rank, see Table \ref{table_e_ranks}.

Figure \ref{fig:BSE_Glycine_reduced} represents similar data as in Figures \ref{fig:BSE_N2H4_reduced},
but for amino-acid Glycine, C$_2$H$_5$NO$_2$, with the BSE matrix size $6000\times 6000$. 
In this case truncation threshold 
$\varepsilon=2\cdot 10^{-1}$ leads to the rank parameters $R_V=54$ , $R_{\overline{W}}=50$, $R_{\widetilde{W}}=50$,
and the error for the minimal eigenvalue, $\omega_1 = 0.2432$ hartree, equals to $0.027$ hartree.   
For $\varepsilon= 10^{-1}$ we have the rank parameters $R_V=100$ , $R_{\overline{W}}=215$, $R_{\widetilde{W}}=129$,
and the error for the minimal eigenvalue equals to $0.014$ hartree ($0.38$eV), while 
the choice $\varepsilon= 10^{-2}$ again ensures the accuracy of the order of $10^{-6}$.

\begin{figure}[htbp]
\centering
\includegraphics[width=5.5cm]{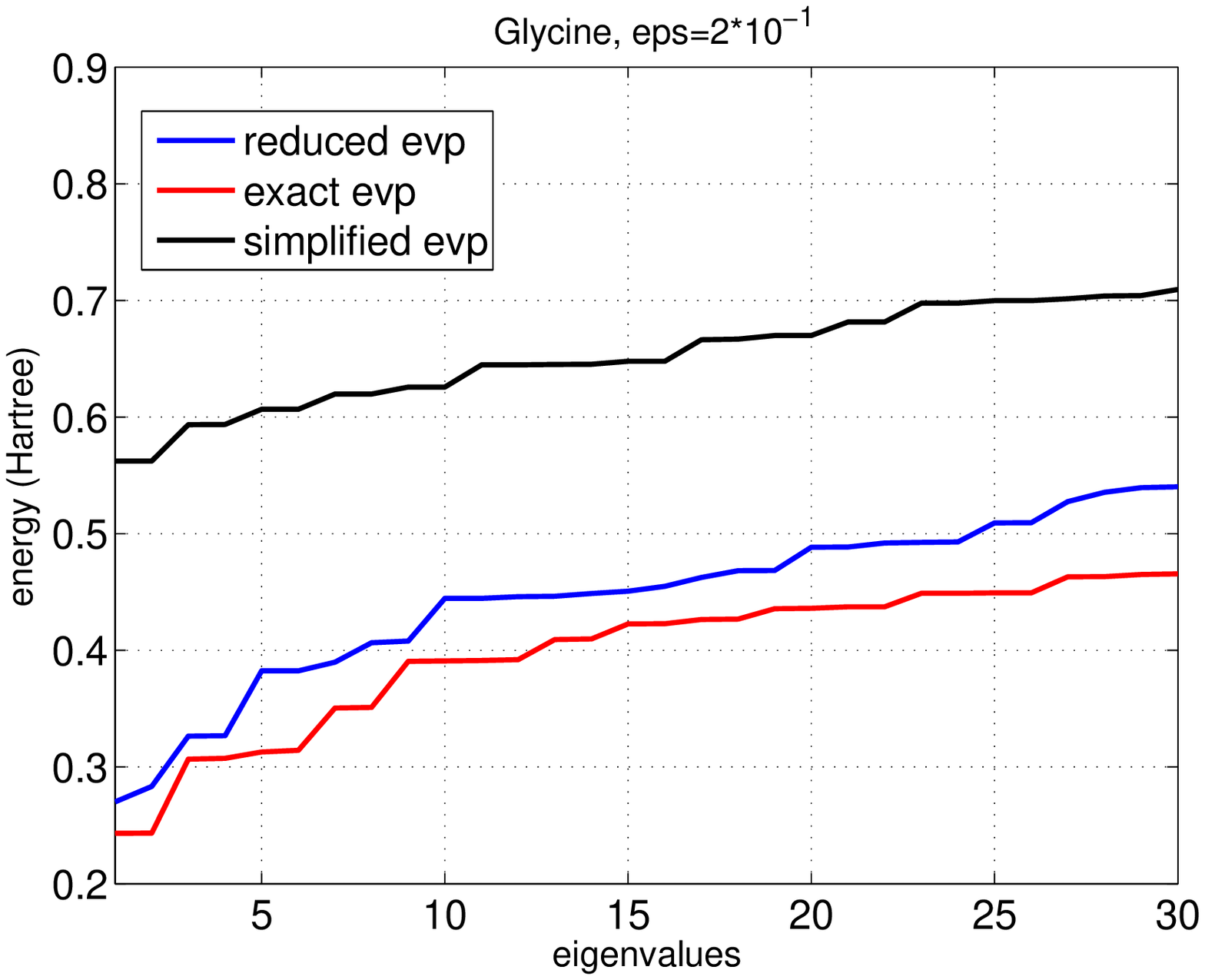} 
\includegraphics[width=5.3cm]{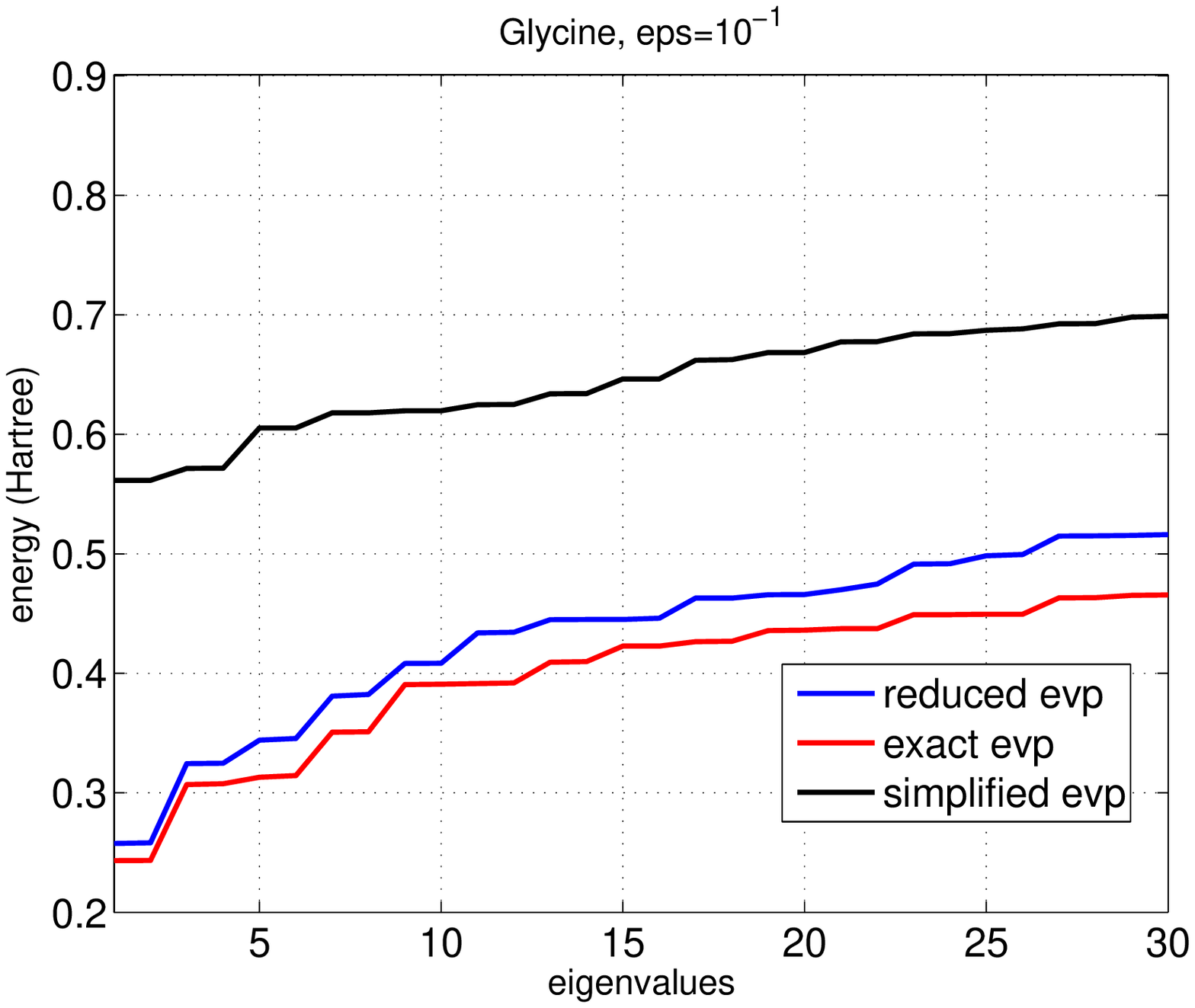}
\includegraphics[width=5.4cm]{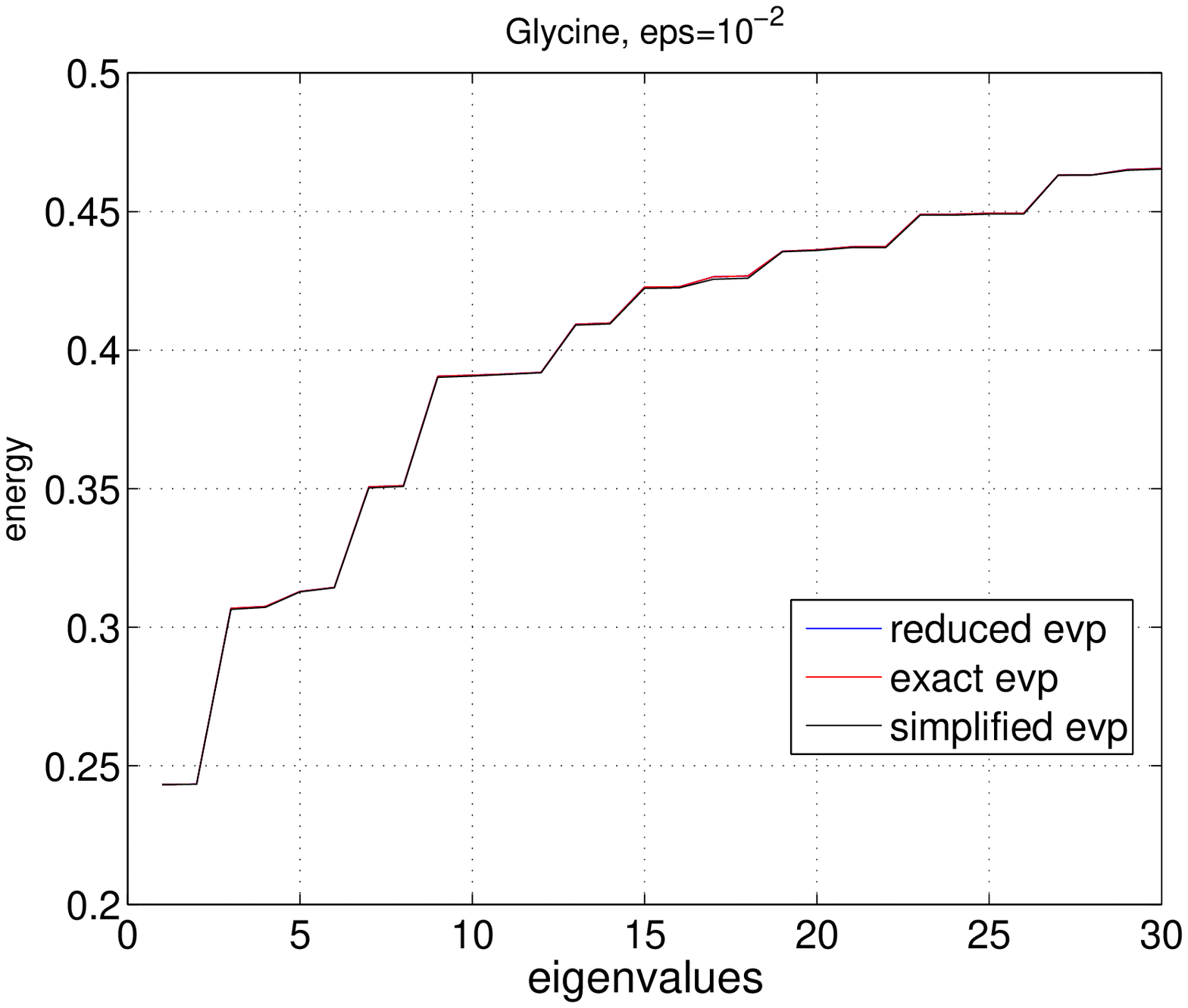}
\caption{\small Comparison of $m_0=30$ lower eigenvalues for the reduced and exact BSE 
systems vs. $\varepsilon$ in the case of Glycine amino acid.}
\label{fig:BSE_Glycine_reduced}  
\end{figure}


Figure \ref{fig:BSE_NH3_reduced}, left, demonstrates the lowest part of the BSE  spectrum
corresponding to the reduced EVP 
in the case of a water molecule, H$_2$O, (rank truncation threshold  $\varepsilon = 10^{-1} $). 
The five lowest values of the computed excitation energy are approximately $8.7$eV, 
$10.7$eV, $11.1$eV, $12.8$eV, and $14.5$eV, see cf. \cite{HerSchmSchw:08}.
Figure \ref{fig:BSE_NH3_reduced}, center and right, illustrates the BSE energy spectrum of the H$_2$O
molecule (based on HF calculations with cc-pDVZ-48 GTO basis)
for the lowest $N_{red}=30$ eigenvalues vs. the rank truncation parameter 
$\varepsilon= 0.6$ and $0.1$, where the ranks of $V$ and the BSE matrix 
block $\widetilde{W}$ are $4$, $5$ and $28$, $30$, respectively, while the block 
$\overline{W}$ remains unchanged. 
For the choice $\varepsilon= 0.6$ and $\varepsilon=0.1$, the error in the 1st (lowest)
eigenvalue for the solution of the problem in reduced basis  is about $0.11$eV  and $0.025$eV, 
correspondingly.
The CPU time in the laptop Matlab implementation of each example is about $5$sec.

\begin{figure}[htbp]
\centering
\includegraphics[width=5.7cm]{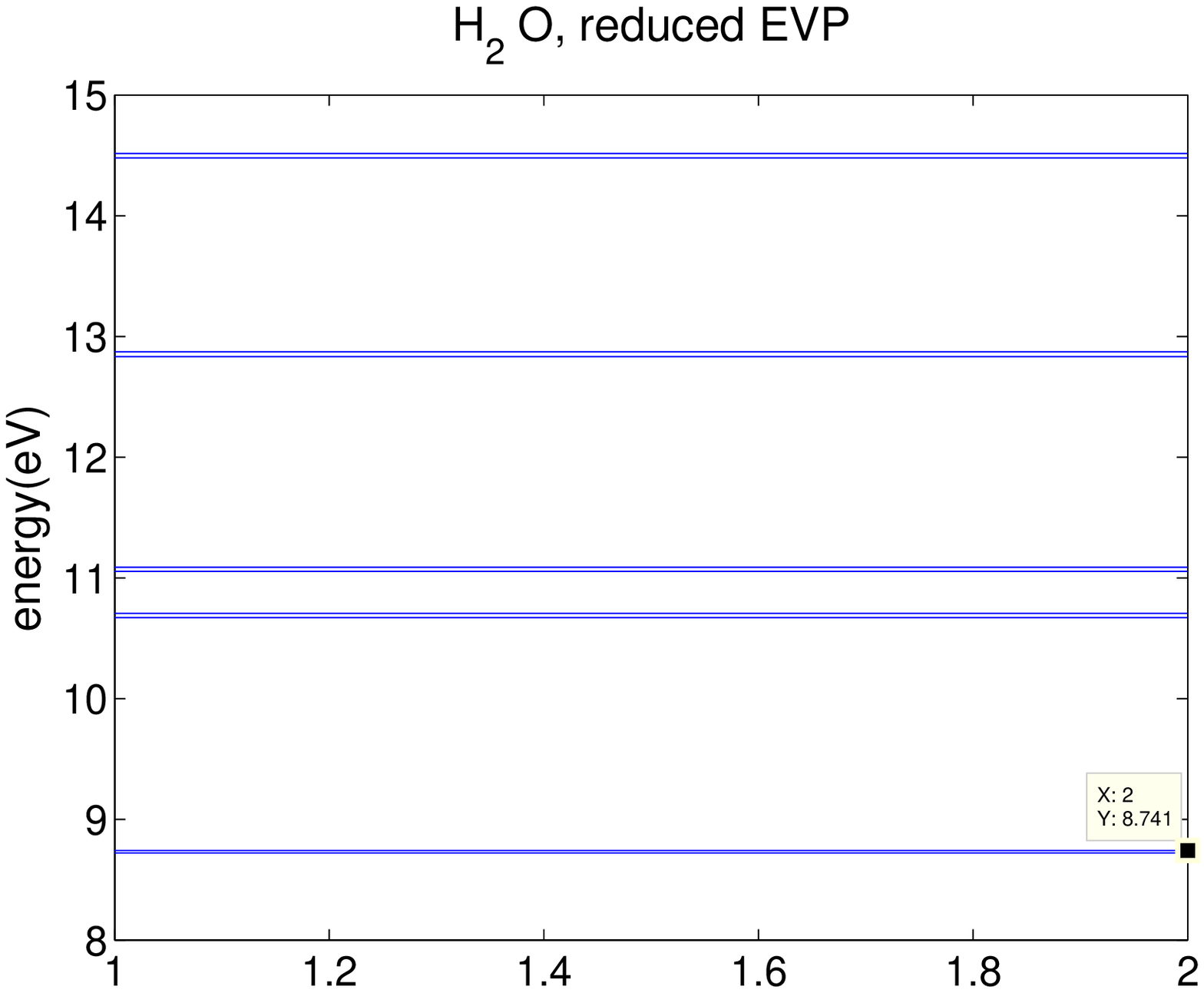}
\includegraphics[width=5.0cm]{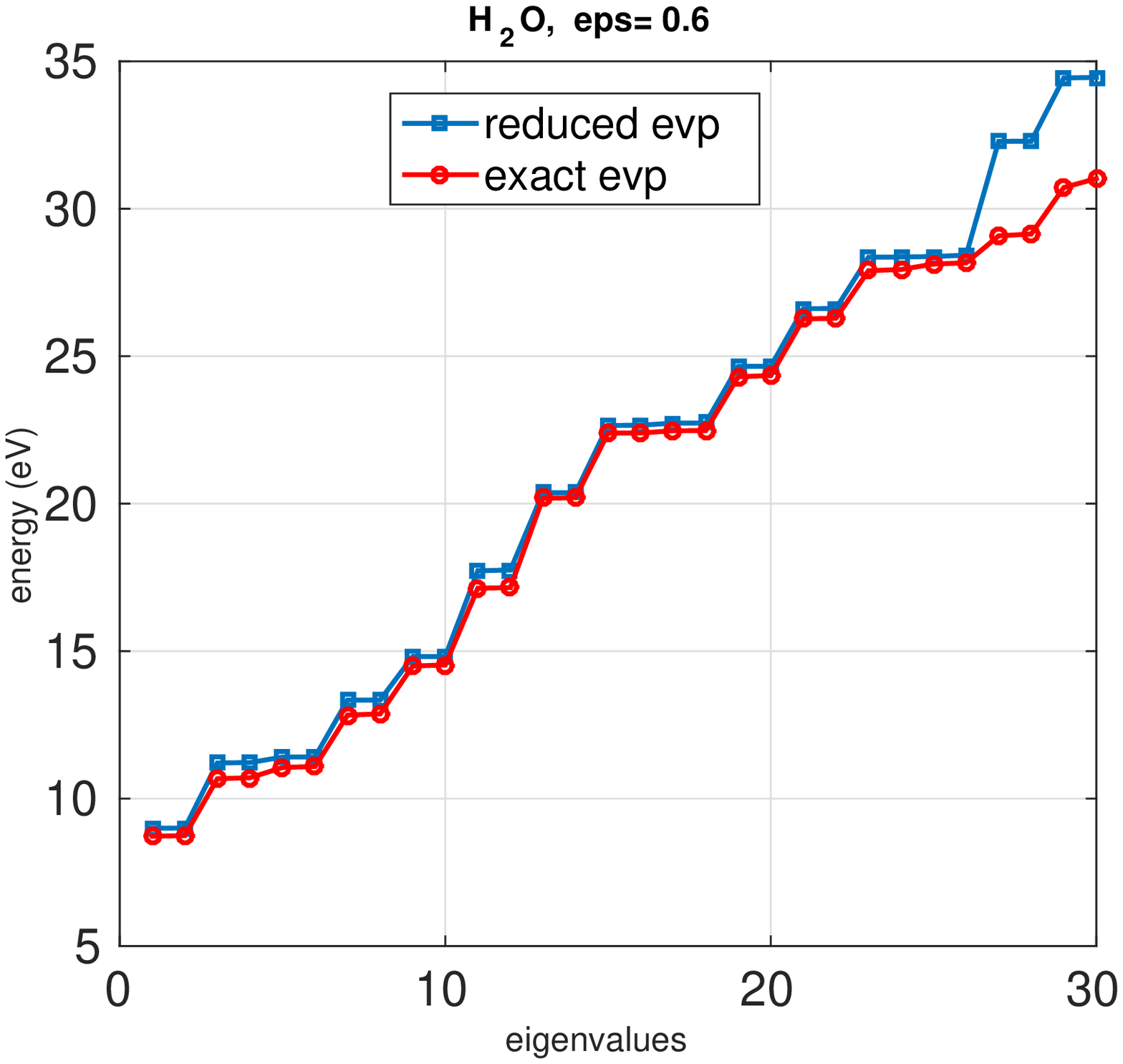} 
\includegraphics[width=5.0cm]{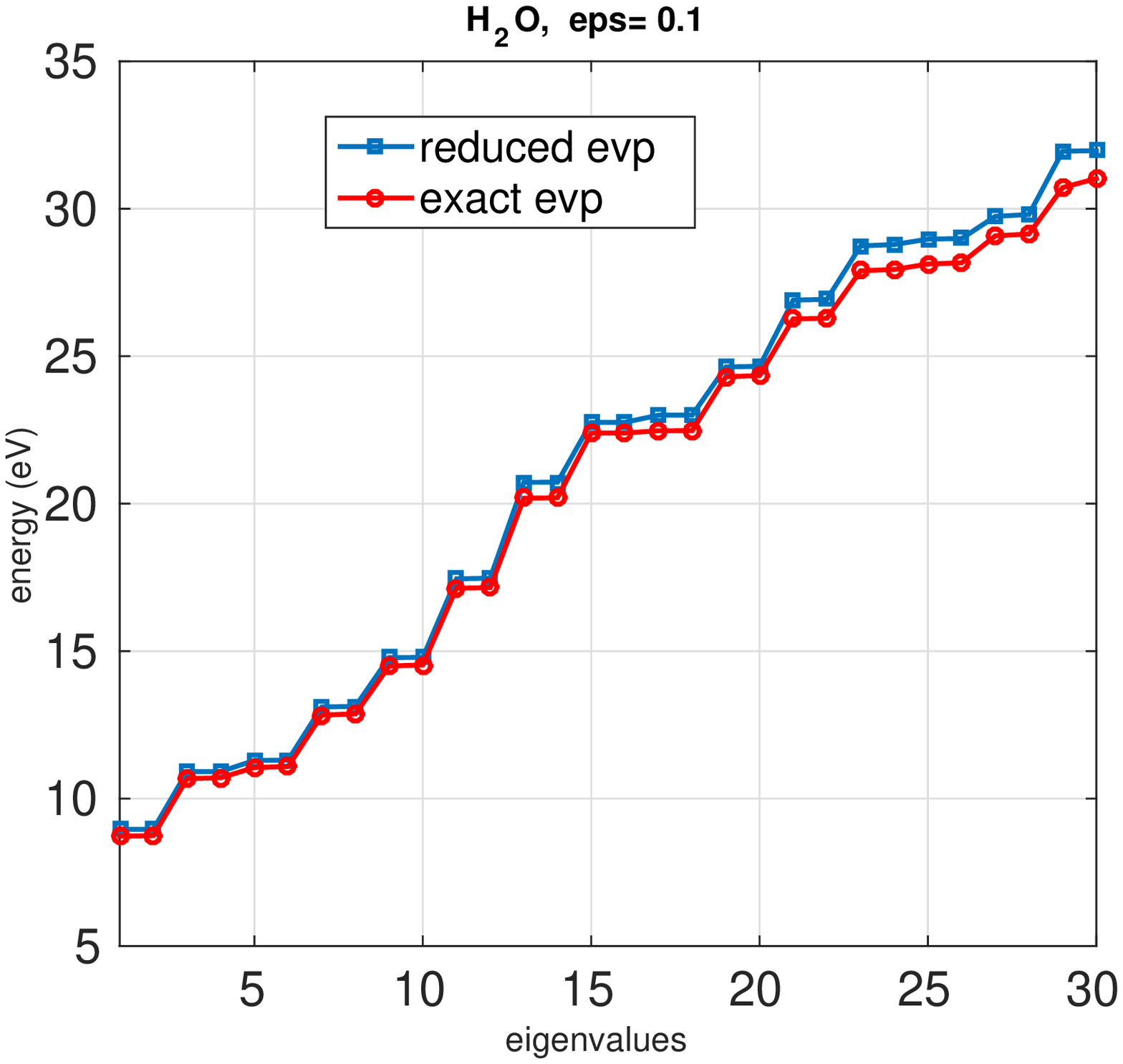}
\caption{\small Comparison of $m_0=30$ lower eigenvalues for the reduced and exact BSE
systems for H$_2$O molecule: $\varepsilon=0.6$, center; $\varepsilon=0.1$, right.}
\label{fig:BSE_NH3_reduced}
\end{figure}

\subsubsection{Comparison with the Tamm-Dancoff approximation} \label{sssec:BSE_TammDanc}

It is interesting to compare the full BSE model with the so-called Tamm-Dancoff approximation (TDA)
\cite{Cas_BSE:95},
which corresponds to setting the matrix $B=0$ in equation (\ref{eqn:BSE-GW1}).
This simplifies the equation (\ref{eqn:BSE-GW1}) to a standard Hermitian eigenvalue problem
\begin{equation} \label{eqn:BSE-Tamm-Danc} 
A {\bf x}_n = \mu_n {\bf x}_n, \quad {\bf x}_n \in \mathbb{R}^{N_{ov}}
\end{equation}
with the reduced matrix size $N_{ov}$. The reduced basis approach via low-rank approximation
described in \S\ref{sssec:Reduced_BSE}
can be applied directly to this equation. 

Below we present numerical tests
indicating that the approximation error introduced by TDA method compared with
the initial BSE system (\ref{eqn:BSE-GW1}) remains on the level of $0.003$ hartree for several
compact molecules, see Figure \ref{fig:BSE_TDA}.  
\begin{figure}[htbp]
\centering
\includegraphics[width=5.0cm]{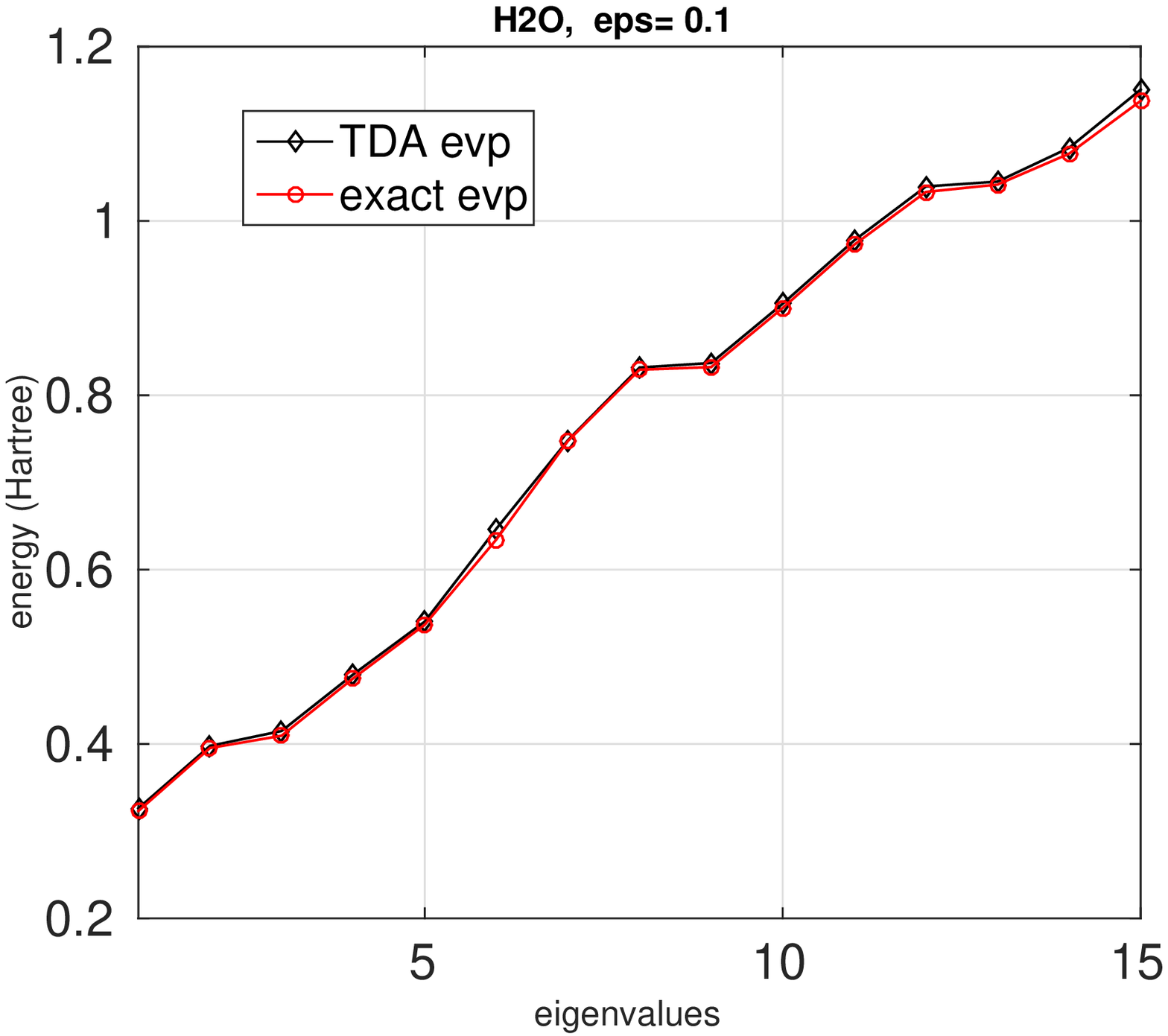}\quad
\includegraphics[width=5.0cm]{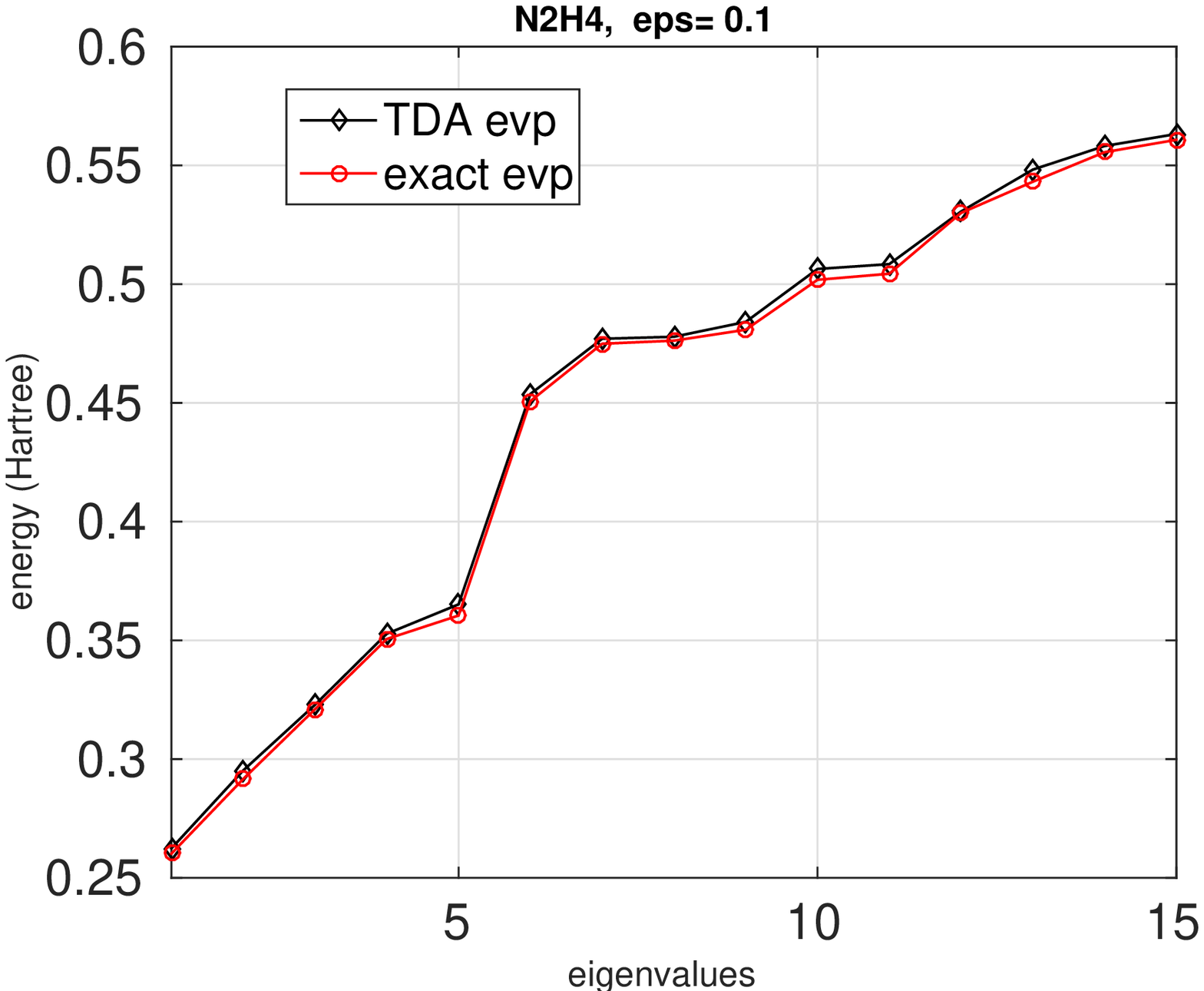} \quad
\includegraphics[width=5.0cm]{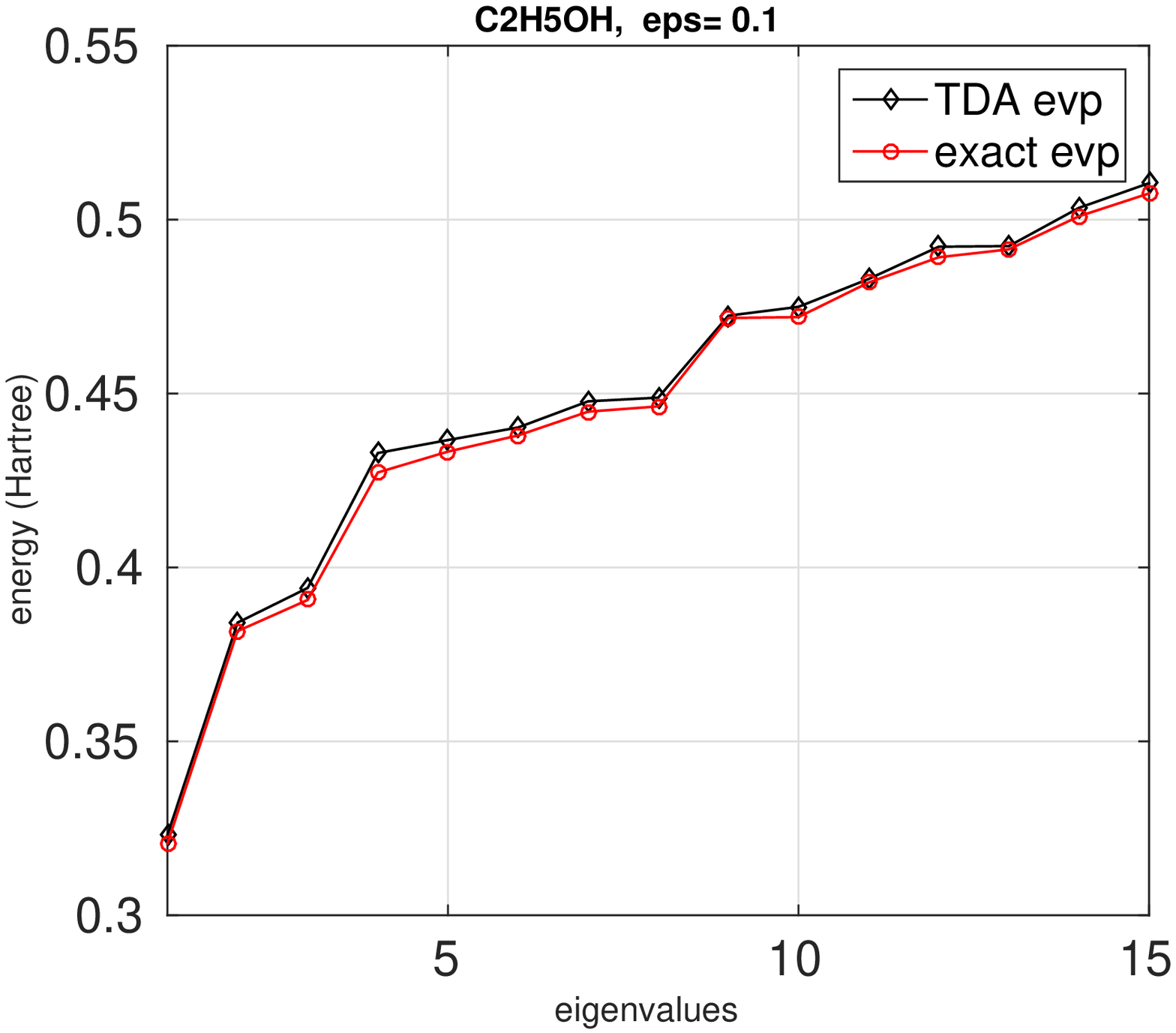}
\caption{\small Comparison between $m_0=30$ lower eigenvalues $\mu_n$ and $\omega_n$
for the TDA and full BSE models, respectively, on the examples of
 H$_2$O, N$_2$H$_4$, and C$_2$H$_5$OH molecules; $\varepsilon=0.1$.}
\label{fig:BSE_TDA}
\end{figure}

Table \ref{table_e_Nred_TDA} indicates a tendency to decrease the TDA model error for larger molecules.

\begin{center}%
\begin{table}[htb]
\begin{center}
\begin{tabular}
[c]{|c|c|r|r|r|r|}%
\hline
 &   H$_2$O  & H$_2$O$_2$ & N$_2$H$_4$ & C$_2$H$_5$OH & Glycine   \\
\hline\hline
err(hartree)& $0.0031$ & $0.0051$ & $0.0022$  & $0.0024$   & $0.0017$      \\
\hline\hline
 \end{tabular}
\end{center}
\caption{The model error $|\mu_1 - \omega_1|$ in TDA approximation for different molecules, $\varepsilon = 0.1$.}
\label{table_e_Nred_TDA}
\end{table}
\end{center}

\section{Conclusions}\label{sec:Conclus}

The new reduced basis method for solving the BSE equation based on the low-rank approximation of matrix blocks
was presented and analyzed. The potential  efficiency of the approach is demonstrated numerically 
on the solution of
large scale Bethe-Salpeter eigenvalue problem for some moderate size molecules and small
amino-acids\footnote{For {\it ab-initio} electronic structure calculations we use the tensor-structured
Hartree-Fock solver \cite{VeKhBoKhSchn:12,vekh:13} implemented in Matlab,
employing the rank-structured calculation of the core Hamiltonian and TEI, using 
the discrete representation of basis 
functions on $n\times n\times n$ 3D Cartesian grids. The arising 3D convolution integrals  
with the Newton kernel are replaced by algebraic operations in 1D complexity.}.
The $\varepsilon$-rank bounds for the requested sub-tensors of the TEI tensor, represented in the 
molecular orbitals (MO) basis set, were proven. 
We justify the quadratic error behavior in the excitation energies with respect
to the accuracy of the rank approximation.  Asymptotic estimates on the storage demands are provided. 

The basic computational scheme of the reduced basis method include:
\begin{enumerate}
 \item Precomputing step I: Given the set of Gaussian type orbitals, compute the related
 TEI tensor in the form of low-rank Cholesky decomposition.
 \item Precomputing step II: Calculate the MO basis set and related energy
 spectrum by solving the Hartree-Fock eigenvalue problem.
 \item Project the TEI tensor onto the MO basis set in the form of low-rank factorization. 
 
\item Compute the diagonal plus low-rank approximations to the 
 matrix blocks $A$ and $B$ and set up the auxiliary eigenvalue problem via rank-structured 
 approximation to the BSE matrix.

 \item Select the reduced basis set from eigenvectors corresponding
 to several lowest eigenstates of the auxiliary structured eigenvalue problem.
 
 \item Compute the Galerkin projection of the exact BSE system matrix onto the reduced basis
 set and solve a small size reduced spectral problem by direct diagonalization. 
\end{enumerate}

We demonstrate that the approximation error of the reduced basis method (1) - (6) 
can be reduced dramatically 
if the matrix block $\overline{W}$ remains unchanged. 
We also analyze the numerical error in the simplified BSE model, 
the so-called Tamm-Dancoff approximation (TDA), specified by the first diagonal 
matrix block $A$.
 
The various numerical tests demonstrate that the reduced basis set obtained by solving 
the auxiliary eigenvalue problem based on the low-rank approximation to BSE matrix blocks 
(with the adaptively chosen rank parameter of the order of several tens) 
allows to achieve the sufficient accuracy for several lowest excited states.
We justify numerically that the simplified TDA equation is characterized by the model error
of the order of $0.003$ hartree ($0.08$ eV) for all molecular systems considered so far,
with the tendency to decrease for large molecules, say $0.0017$ hartree ($0.045$ eV) for Glycine amino acid.
We note in closing that here we mainly focus on the numerical efficiency of the new 
computational scheme with respect to the accuracy vs. separation rank, tested on the
Hartree-Fock-BSE and Hartree-Fock-BSE-TDA calculations for some moderate size molecules. 

The future work is concerned with the design of the efficient linear algebra
algorithms for fast solution of arising large eigenvalue problems with diagonal plus rank-structured 
matrices.
The approach can be also extended to the case of finite non-periodic lattice systems 
(e.g. quantum dots or nanoparticles)
providing gainful opportunities for data-sparse matrix calculus.

Another possible direction includes the quantized tensor approximation (QTT) \cite{KhQuant:09}
of the matrices involved in order to perform the super-fast matrix-vector calculations
in the QTT tensor arithmetics with the $\varepsilon$-rank truncation (see e.g. \cite{DoKhSavOs_mEIG:13}).

\vspace{0.5cm}

{\bf Acknowledgements.} The authors would like to acknowledge Prof. A. Savin (UPMC, Paris)
and Prof. J. Toulouse (UPMC, Paris) for valuable comments on the problem setting 
for the BSE model and for providing the useful references.

\section{Appendix: The Hartree-Fock model}\label{sec:Append}


The $2N_{orb}$-electrons Hartree-Fock equation for pairwise $L^2$-orthogonal
electronic orbitals, $\psi_i: \mathbb{R}^3\to\mathbb{R} $,
$\psi_i\in H^1(\mathbb{R}^3)$, reads as
\begin{equation} \label{HaFo eq.}
  {\cal F}  \psi_i({ x}) = \lambda_i \, \psi_i({ x}),\quad
  \int_{\mathbb{R}^3}\psi_i\psi_j  dx =\delta_{ij},
   \;\; i,j=1,...,N_{orb},
\end{equation}
with ${\cal F}$ being the nonlinear  Fock operator 
\[
{\cal F} :=-\frac{1}{2} \Delta + V_c + V_H + {\cal K}.
\]
Here the nuclear potential takes the form  
\[
V_c(x)=- \sum_{\nu=1}^{M}\frac{Z_\nu}{\|{x} -a_\nu \|},\quad
Z_\nu >0, \;\; a_\nu\in \mathbb{R}^3,
\]
while the Hartree potential $V_H({ x})$ and the nonlocal exchange operator $\cal K$ read as
 \begin{equation}
   V_{H}({x}):= \rho \star \frac{1}{\|\cdot\|} =  \int_{\mathbb{R}^3}
 \frac{\rho({y})}{\|{x}-{y}\|}\, d{y} , \quad{x}\in {\mathbb{R}^3},
 \label{Reference_HP}
\end{equation}
and
\begin{equation}
\left( {\cal K}\psi\right) ({ x}):=- 
\sum_{i=1}^{N_{orb}}\left(\psi \, \psi_i\star \frac{1}{\|\cdot\|} \right)
\psi_i (x)=
 - \frac{1}{2}\int_{\mathbb{R}^3} \; 
\frac{\tau({ x}, { y})}{\|{x} - { y}\|}\,\psi({y}) d{ y},
\label{VExch}
\end{equation} 
respectively. Conventionally, we use the definitions
\begin{equation*} \label{DMatr-eq}
\tau({ x}, { y}) := 2\sum_{i=1}^{N_{orb}} \; \psi_i({ x}) \psi_i({ y}), 
\quad  \rho( {x}) :=\tau({x},{x}),\quad
\end{equation*}
for the density matrix $\tau({x}, { y})$, and electron density $\rho({x})$.

Usually, the Hartree-Fock equation is approximated by the standard Galerkin projection of the 
initial problem  (\ref{HaFo eq.}) posed in $H^1(\mathbb{R}^3)$.
For a given finite Galerkin basis set $\{g_\mu \}_{1\leq \mu \leq N_b}$,
$g_\mu\in H^1(\mathbb{R}^3) $, the occupied
molecular orbitals $\psi_i$ are represented (approximately) as
 \begin{equation}\label{expand}
\psi_i=\sum\limits_{\mu=1}^{N_b} C_{\mu i} g_\mu, \quad i=1,...,N_{orb}.
\end{equation}
To derive an equation for the unknown coefficients matrix
$C=\{C_{ \mu i} \}\in \mathbb{R}^{N_b \times  N_{orb}}$, first, we introduce
the mass (overlap) matrix $S=\{S_{\mu \nu} \}_{1\leq \mu, \nu \leq N_b}$, given by
\[
S_{\mu \nu}=\int_{\mathbb{R}^3} g_\mu g_\nu  dx,
\]
and the stiffness matrix $H=\{h_{\mu \nu}\}$ of the core Hamiltonian
${\cal H}=-\frac{1}{2} \Delta + V_c $ (the single-electron integrals),
\[
h_{\mu \nu}= \frac{1}{2} \int_{\mathbb{R}^3}\nabla g_\mu \cdot \nabla g_\nu dx +
\int_{\mathbb{R}^3} V_c(x) g_\mu g_\nu dx, \quad 1\leq \mu, \nu \leq N_b.
\]
The core Hamiltonian matrix $H$ can be precomputed in $O(N_b^2)$ operations via grid-based approach.

Given the finite basis set $\{g_\mu \}_{1\leq \mu \leq N_b}$, $g_\mu\in H^1(\mathbb{R}^3) $, the
associated fourth order two-electron integrals (TEI) tensor,
${\textbf{B}}=[b_{\mu \nu \lambda \sigma}]$, is defined entrywise by
\begin{equation} \label{eqn:btensor}
b_{\mu \nu \lambda \sigma}= \int_{\mathbb{R}^3}\int_{\mathbb{R}^3}
\frac{g_\mu(x) g_\nu(x) g_\lambda(y) g_\sigma(y) }{\| x-y \|} dx dy, \quad
\mu, \nu, \lambda, \sigma \in \{1,...,N_b\}=:{\cal I}_b.
\end{equation}

In computational quantum chemistry the nonlinear terms representing the Galerkin 
approximation to the Hartree and
exchange operators are calculated traditionally by using the low-rank Cholesky decomposition of the
TEI tensor ${\bf B}=[b_{\mu \nu \kappa \lambda}]$ as defined in (\ref{eqn:btensor}),
(\ref{eqn:BCholesky}) that initially has the computational and storage complexity  of order $O(N_b^4)$.
 
Introducing the $N_b\times N_b$ matrices $J(D)$ and $K(D)$,
\begin{equation}\label{C-K matr}
J(D)_{\mu \nu}= \sum\limits_{\kappa, \lambda=1}^{N_b}
b_{\mu \nu, \kappa \lambda}D_{\kappa \lambda},\quad
K(D)_{\mu \nu}= - \frac{1}{2} \sum\limits_{\kappa, \lambda=1}^{N_b}
b_{\mu \lambda, \nu \kappa}D_{\kappa \lambda},
\end{equation}
where  $D= 2 C C^T \in \mathbb{R}^{N_b\times N_b}$ is the rank-$N_{orb}$ symmetric density matrix,
one then represents the complete Fock matrix $F$ by
\begin{equation}\label{Fock_matr}
 F(D)= H+ J(D) + K(D).   
\end{equation}

The resultant Galerkin system of nonlinear equations
for the coefficients matrix $C\in \mathbb{R}^{N_b \times N_{orb}}$, 
and the respective eigenvalues $\Lambda$, reads as
\begin{align} \label{HF discr}
  F(D) C &= SC \Lambda, \quad \Lambda= diag(\lambda_1,...,\lambda_{N_b}), \\
  C^T SC   &=  I_N,   \nonumber
\end{align}
where the second equation represents the orthogonality constraints
$\int_{\mathbb{R}^3}\psi_i\psi_j dx=\delta_{ij}$, and
$I_N$ denotes the $N_b\times N_b $ identity matrix.

\begin{footnotesize}

\bibliographystyle{abbrv}
\bibliography{BSE_Fock.bib}
\end{footnotesize}
\end{document}